\documentclass[runningheads, a4paper]{llncs}

\usepackage[utf8]{inputenc}
\usepackage[T1]{fontenc}
\usepackage[british]{babel}

\usepackage{textcomp}
\usepackage{scrhack}
\usepackage{xspace}
\usepackage{mparhack}
\usepackage{fixltx2e}

\usepackage{amsmath}
\usepackage[warning, all]{onlyamsmath}\AtBeginDocument{\catcode`\$=3}
\usepackage{amssymb}
\usepackage{stmaryrd}
\usepackage{dsfont}
\usepackage{mathtools}
\usepackage{mathdots}
\usepackage{gensymb}
\allowdisplaybreaks

\usepackage{graphicx}
\usepackage{subfig}
\usepackage{scalerel}
\usepackage{csquotes}
\usepackage{url}
\usepackage{adjustbox}
\usepackage{tabularx}

\usepackage{caption}
\usepackage{tikz}
\usetikzlibrary{positioning, arrows, patterns}

\spnewtheorem{main-theorem}{Main Theorem}{\bfseries}{\itshape}
\spnewtheorem*{proof-idea}{Proof Idea}{\itshape}{\rmfamily}
\spnewtheorem*{intuition}{Intuition}{\itshape}{\rmfamily}

\newtheorem{lemmaxinner}{Lemma}
\newenvironment{lemmax}[1]
  {\renewcommand\thelemmaxinner{#1}\lemmaxinner}
  {\endlemmaxinner}

\urldef{\mails}\path|simon.wacker@kit.edu|
\newcommand{\keywords}[1]{\par\addvspace\baselineskip\noindent\keywordname\enspace\ignorespaces#1}

\newcommand*{\define}[1]{\emph{#1}}

\providecommand\iff{\DOTSB\;\Longleftrightarrow\;}
\providecommand\implies{\DOTSB\;\Longrightarrow\;}

\DeclarePairedDelimiter\parens{\lparen}{\rparen}

\DeclarePairedDelimiter\abs{\lvert}{\rvert}

\DeclarePairedDelimiterX\inner[2]{\langle}{\rangle}{#1,#2} 
\DeclarePairedDelimiter\set{\{}{\}} 
\DeclarePairedDelimiter\net{\{}{\}}
\DeclarePairedDelimiter\family{\{}{\}}

\mathchardef\breakingcomma\mathcode`\,
{\catcode`,=\active
  \gdef,{\breakingcomma\discretionary{}{}{}}
}
\newcommand*\ntuple[1]{\lparen\mathcode`\,=\string"8000 #1\rparen}

\newcommand*{\Z}{{\mathbb{Z}}}
\newcommand*{\R}{\mathbb{R}}

\DeclareMathOperator{\identity}{id}

\DeclareMathOperator{\Exists}{\exists}
\DeclareMathOperator{\ForEach}{\forall}
\newcommand*\Holds{:}
\newcommand*\SuchThat{:}

\newcommand*\leftaction{\triangleright}
\newcommand*\inducedleftaction{\mathbin{\protect\scalerel*{\blacktriangleright}{\triangleright}}}
\newcommand*\rightsemiaction{\mathbin{\protect\scalerel*{\trianglelefteqslant}{\rhd}}} 
\newcommand*\inducedrightsemiaction{\mathbin{\protect\scalerel*{\blacktriangleleft}{\triangleleft}}} 

\DeclareMathOperator{\Sym}{Sym}

\newcommand*{\transpose}{\intercal} 

\makeatletter
\def\moverlay{\mathpalette\mov@rlay}
\def\mov@rlay#1#2{\leavevmode\vtop{%
   \baselineskip\z@skip \lineskiplimit-\maxdimen
   \ialign{\hfil$\m@th#1##$\hfil\cr#2\crcr}}}
\newcommand{\charfusion}[3][\mathord]{
    #1{\ifx#1\mathop\vphantom{#2}\fi
        \mathpalette\mov@rlay{#2\cr#3}
      }
    \ifx#1\mathop\expandafter\displaylimits\fi}
\makeatother
\newcommand{\cupdot}{\charfusion[\mathbin]{\cup}{\cdot}}

\newcommand*\boundary{\partial}
\DeclareMathOperator{\entropy}{ent}
\DeclareMathOperator{\powerset}{\mathcal{P}}
\DeclareMathOperator{\domain}{dom}
\DeclareMathOperator{\diff}{diff}
\DeclareMathOperator{\Cyl}{Cyl}
\newcommand*\occurs{\sqsubseteq}

\newcommand*\suchthat{\mid} 
\newcommand*\from{\colon} 

\newcommand*\quotient{\slash}
\renewcommand*{\restriction}{\mathord{\upharpoonright}}

\newcommand*{\blank}{\mathord{\_}} 
\newcommand{\closure}[1]{\mkern 1.5mu\overline{\mkern-1.5mu#1\mkern-1.5mu}\mkern 1.5mu} 

\newcommand*{\graffito}[1]{}
\newcommand*{\mathnote}[1]{}
\renewcommand*{\index}[1]{}

\begin{document}

  \mainmatter

  \title{The Garden of Eden Theorem for Cellular Automata on Group Sets}
  \authorrunning{The Garden of Eden Theorem for Cellular Automata on Group Sets}
  \author{Simon Wacker}

  \institute{%
    Karlsruhe Institute of Technology\\
    \mails\\
    \url{http://www.kit.edu}%
  }

  \maketitle

  \begin{abstract}
    We prove the Garden of Eden theorem for cellular automata with finite set of states and finite neighbourhood on right amenable left homogeneous spaces with finite stabilisers. It states that the global transition function of such an automaton is surjective if and only if it is pre-injective. Pre-Injectivity means that two global configurations that differ at most on a finite subset and have the same image under the global transition function must be identical.
  \end{abstract} 

  The notion of amenability for groups was introduced by John von Neumann in 1929. It generalises the notion of finiteness. A group $G$ is \emph{left} or \emph{right amenable} if there is a finitely additive probability measure on $\powerset(G)$ that is invariant under left and right multiplication respectively. Groups are left amenable if and only if they are right amenable. A group is \emph{amenable} if it is left or right amenable.

  The definitions of left and right amenability generalise to left and right group sets respectively. A left group set $\ntuple{M, G, \leftaction}$ is \emph{left amenable} if there is a finitely additive probability measure on $\powerset(M)$ that is invariant under $\leftaction$. There is in general no natural action on the right that is to a left group action what right multiplication is to left group multiplication. Therefore, for a left group set there is no natural notion of right amenability.

  A transitive left group action $\leftaction$ of $G$ on $M$ induces, for each element $m_0 \in M$ and each family $\family{g_{m_0, m}}_{m \in M}$ of elements in $G$ such that, for each point $m \in M$, we have $g_{m_0, m} \leftaction m_0 = m$, a right quotient set semi-action $\rightsemiaction$ of $G \quotient G_0$ on $M$ with defect $G_0$ given by $m \rightsemiaction g G_0 = g_{m_0, m} g g_{m_0, m}^{-1} \leftaction m$, where $G_0$ is the stabiliser of $m_0$ under $\leftaction$. Each of these right semi-actions is to the left group action what right multiplication is to left group multiplication. They occur in the definition of global transition functions of cellular automata over left homogeneous spaces as defined in \cite{wacker:automata:2016}. A \emph{cell space} is a left group set together with choices of $m_0$ and $\family{g_{m_0, m}}_{m \in M}$.

  A cell space $\mathcal{R}$ is \emph{right amenable} if there is a finitely additive probability measure on $\powerset(M)$ that is semi-invariant under $\rightsemiaction$. For example cell spaces with finite sets of cells, abelian groups, and finitely right generated cell spaces of sub-exponential growth are right amenable, in particular, quotients of finitely generated groups of sub-exponential growth by finite subgroups acted on by left multiplication. A net of non-empty and finite subsets of $M$ is a \emph{right Følner net} if, broadly speaking, these subsets are asymptotically invariant under $\rightsemiaction$. A finite subset $E$ of $G \quotient G_0$ and two partitions $\family{A_e}_{e \in E}$ and $\family{B_e}_{e \in E}$ of $M$ constitute a \emph{right paradoxical decomposition} if the map $\blank \rightsemiaction e$ is injective on $A_e$ and $B_e$, and the family $\family{(A_e \rightsemiaction e) \cupdot (B_e \rightsemiaction e)}_{e \in E}$ is a partition of $M$. The Tarski-Følner theorem states that right amenability, the existence of right Følner nets, and the non-existence of right paradoxical decompositions are equivalent. We prove it in \cite{wacker:amenable:2016} for cell spaces with finite stabilisers.

  For a right amenable cell space with finite stabilisers we may choose a right Følner net $\mathcal{F} = \family{F_i}_{i \in I}$. The entropy of a subset $X$ of $Q^M$ with respect to $\mathcal{F}$, where $Q$ is a finite set, is, broadly speaking, the asymptotic growth rate of the number of finite patterns with domain $F_i$ that occur in $X$. For subsets $E$ and $E'$ of $G \quotient G_0$, an $\ntuple{E, E'}$-tiling is a subset $T$ of $M$ such that $\family{t \rightsemiaction E}_{t \in T}$ is pairwise disjoint and $\family{t \rightsemiaction E'}_{t \in T}$ is a cover of $M$. If for each point $t \in T$ not all patterns with domain $t \rightsemiaction E$ occur in a subset of $Q^M$, then that subset does not have maximal entropy. 

  The global transition function of a cellular automaton with finite set of states and finite neighbourhood over a right amenable cell space with finite stabilisers, as introduced in \cite{wacker:automata:2016}, is surjective if and only if its image has maximal entropy
  and it is pre-injective if and only if its image has maximal entropy.
  This establishes the Garden of Eden theorem,
  which states that a global transition function as above is surjective if and only if it is pre-injective. This answers a question posed by S{\'{e}}bastien Moriceau at the end of his paper \enquote{Cellular Automata on a $G$-Set}\cite{moriceau:2011}.

  The Garden of Eden theorem for cellular automata over $\Z^2$ is a famous theorem by Edward Forrest Moore and John R. Myhill from 1962 and 1963, see the papers \enquote{Machine models of self-reproduction}\cite{moore:1962} and \enquote{The converse of Moore's Garden-of-Eden theorem}\cite{myhill:1963}. This paper is greatly inspired by the monograph \enquote{Cellular Automata and Groups}\cite{ceccherini-silberstein:coornaert:2010} by Tullio Ceccherini-Silberstein and Michel Coornaert. 

  In Sect.~\ref{sec:interiors-closures-and-boundaries} we introduce $E$-interiors, $E$-closures, and $E$-boundaries of subsets of $M$. In Sect.~\ref{sec:tilings} we introduce $\ntuple{E, E'}$-tilings of cell spaces. In Sect.~\ref{sec:entropies} we introduce entropies of subsets of $Q^M$. And in Sect.~\ref{sec:gardens-of-eden} we prove the Garden of Eden theorem.

  \subsubsection{Preliminary Notions.} A \define{left group set} is a triple $\ntuple{M, G, \leftaction}$, where $M$ is a set, $G$ is a group, and $\leftaction$ is a map from $G \times M$ to $M$, called \define{left group action of $G$ on $M$}, such that $G \to \Sym(M)$, $g \mapsto [g \leftaction \blank]$, is a group homomorphism. The action $\leftaction$ is \define{transitive} if $M$ is non-empty and for each $m \in M$ the map $\blank \leftaction m$ is surjective; and \define{free} if for each $m \in M$ the map $\blank \leftaction m$ is injective. For each $m \in M$, the set $G \leftaction m$ is the \define{orbit of $m$}, the set $G_m = (\blank \leftaction m)^{-1}(m)$ is the \define{stabiliser of $m$}, and, for each $m' \in M$, the set $G_{m, m'} = (\blank \leftaction m)^{-1}(m')$ is the \define{transporter of $m$ to $m'$}.

  A \define{left homogeneous space} is a left group set $\mathcal{M} = \ntuple{M, G, \leftaction}$ such that $\leftaction$ is transitive. A \define{coordinate system for $\mathcal{M}$} is a tuple $\mathcal{K} = \ntuple{m_0, \family{g_{m_0, m}}_{m \in M}}$, where $m_0 \in M$ and for each $m \in M$ we have $g_{m_0, m} \leftaction m_0 = m$. The stabiliser $G_{m_0}$ is denoted by $G_0$. The tuple $\mathcal{R} = \ntuple{\mathcal{M}, \mathcal{K}}$ is a \define{cell space}. The set $\set{g G_0 \suchthat g \in G}$ of left cosets of $G_0$ in $G$ is denoted by $G \quotient G_0$. The map $\rightsemiaction \from M \times G \quotient G_0 \to M$, $(m, g G_0) \mapsto g_{m_0, m} g g_{m_0, m}^{-1} \leftaction m\ (= g_{m_0, m} g \leftaction m_0)$ is a \define{right semi-action of $G \quotient G_0$ on $M$ with defect $G_0$}, which means that
  \begin{gather*}
    \ForEach m \in M \Holds m \rightsemiaction G_0 = m,\\
    \ForEach m \in M \ForEach g \in G \Exists g_0 \in G_0 \SuchThat \ForEach \mathfrak{g}' \in G \quotient G_0 \Holds
          m \rightsemiaction g \cdot \mathfrak{g}' = (m \rightsemiaction g G_0) \rightsemiaction g_0 \cdot \mathfrak{g}'.
  \end{gather*}
  It is \define{transitive}, which means that the set $M$ is non-empty and for each $m \in M$ the map $m \rightsemiaction \blank$ is surjective; and \define{free}, which means that for each $m \in M$ the map $m \rightsemiaction \blank$ is injective; and \define{semi-commutes with $\leftaction$}, which means that
  \begin{equation*}
    \ForEach m \in M \ForEach g \in G \Exists g_0 \in G_0 \SuchThat \ForEach \mathfrak{g}' \in G \quotient G_0 \Holds
          (g \leftaction m) \rightsemiaction \mathfrak{g}' = g \leftaction (m \rightsemiaction g_0 \cdot \mathfrak{g}').
  \end{equation*}
  The maps $\iota \from M \to G \quotient G_0$, $m \mapsto G_{m_0, m}$, and $m_0 \rightsemiaction \blank$ are inverse to each other. Under the identification of $M$ with $G \quotient G_0$ by either of these maps, we have $\rightsemiaction \from (m, \mathfrak{g}) \mapsto g_{m_0, m} \leftaction \mathfrak{g}$.

  A left homogeneous space $\mathcal{M}$ is \define{right amenable} if there is a coordinate system $\mathcal{K}$ for $\mathcal{M}$ and there is a finitely additive probability measure $\mu$ on $M$ such that 
  \begin{equation*}
    \ForEach \mathfrak{g} \in G \quotient G_0 \ForEach A \subseteq M \Holds \parens[\big]{(\blank \rightsemiaction \mathfrak{g})\restriction_A \text{ injective} \implies \mu(A \rightsemiaction \mathfrak{g}) = \mu(A)},
  \end{equation*}
  in which case the cell space $\mathcal{R} = \ntuple{\mathcal{M}, \mathcal{K}}$ is called \define{right amenable}. When the stabiliser $G_0$ is finite, that is the case if and only if there is a \define{right Følner net in $\mathcal{R}$ indexed by $(I, \leq)$}, which is a net $\net{F_i}_{i \in I}$ in $\set{F \subseteq M \suchthat F \neq \emptyset, F \text{ finite}}$ such that
  \begin{equation*}
    \ForEach \mathfrak{g} \in G \quotient G_0 \Holds \lim_{i \in I} \frac{\abs{F_i \smallsetminus (\blank \rightsemiaction \mathfrak{g})^{-1}(F_i)}}{\abs{F_i}} = 0.
  \end{equation*}

  A \define{semi-cellular automaton} is a quadruple $\mathcal{C} = \ntuple{\mathcal{R}, Q, N, \delta}$, where $\mathcal{R}$ is a cell space; $Q$, called \define{set of states}, is a set; $N$, called \define{neighbourhood}, is a subset of $G \quotient G_0$ such that $G_0 \cdot N \subseteq N$; and $\delta$, called \define{local transition function}, is a map from $Q^N$ to $Q$. A \define{local configuration} is a map $\ell \in Q^N$, a \define{global configuration} is a map $c \in Q^M$, and a \define{pattern} is a map $p \in Q^A$, where $A$ is a subset of $M$. The stabiliser $G_0$ acts on $Q^N$ on the left by $\bullet \from G_0 \times Q^N \to Q^N$, $(g_0, \ell) \mapsto [n \mapsto \ell(g_0^{-1} \cdot n)]$, and the group $G$ acts on the set of patterns on the left by
  \begin{align*}
    \inducedleftaction \from G \times \bigcup_{A \subseteq M} Q^A &\to     \bigcup_{A \subseteq M} Q^A,\\
                                                           (g, p) &\mapsto \left[
                                                                             \begin{aligned}
                                                                               g \leftaction \domain(p) &\to     Q,\\
                                                                                                      m &\mapsto p(g^{-1} \leftaction m).
                                                                             \end{aligned}
                                                                           \right]
  \end{align*}
%
  The \define{global transition function of $\mathcal{C}$} is the map $\Delta \from Q^M \to Q^M$, $c \mapsto [m \mapsto \delta(n \mapsto c(m \rightsemiaction n))]$.

  A \define{cellular automaton} is a semi-cellular automaton $\mathcal{C} = \ntuple{\mathcal{R}, Q, N, \delta}$ such that $\delta$ is \define{$\bullet$-invariant}, which means that, for each $g_0 \in G_0$, we have $\delta(g_0 \bullet \blank) = \delta(\blank)$. Its global transition function is $\inducedleftaction$-equivariant, which means that, for each $g \in G$, we have $\Delta(g \inducedleftaction \blank) = g \inducedleftaction \Delta(\blank)$.

  For each $A \subseteq M$, let $\pi_A \from Q^M \to Q^A$, $c \mapsto c\restriction_A$.

  \section{Interiors, Closures, and Boundaries} 
  \label{sec:interiors-closures-and-boundaries}

  In this section, let $\mathcal{R} = \ntuple{\ntuple{M, G, \leftaction}, \ntuple{m_0, \family{g_{m_0, m}}_{m \in M}}}$ be a cell space.

  In Definition~\ref{def:interior-closure-boundary} we introduce $E$-interiors, $E$-closures, and $E$-boundaries of subsets of $M$. In Lemma~\ref{lem:Delta-X-A-minus-plus-are-surjective} we define surjective restrictions $\Delta_{X, A}^-$ of global transition functions to patterns. And in Theorem~\ref{thm:boundary-characterisation-of-folner-net} we show that right Følner nets are those nets whose components are asymptotically invariant under taking finite boundaries.

  \begin{definition} 
  \label{def:interior-closure-boundary}
    Let $A$ be a subset of $M$ and let $E$ be a subset of $G \quotient G_0$.
    \begin{enumerate}
      \item The set 
            \begin{equation*}
              A^{-E} = \set{m \in M \suchthat m \rightsemiaction E \subseteq A}\ \parens[\big]{= \bigcap_{e \in E} \bigcup_{a \in A} (\blank \rightsemiaction e)^{-1}(a)} \mathnote{$E$-interior $A^{-E}$ of $A$}
            \end{equation*}
            is called \define{$E$-interior of $A$}\index{interior of $A$@$E$-interior of $A$}.
      \item The set
            \begin{equation*} 
              A^{+E} = \set{m \in M \suchthat (m \rightsemiaction E) \cap A \neq \emptyset}\ \parens[\big]{= \bigcup_{e \in E} \bigcup_{a \in A} (\blank \rightsemiaction e)^{-1}(a)} \mathnote{$E$-closure $A^{+E}$ of $A$}
            \end{equation*}
            is called \define{$E$-closure of $A$}\index{closure of $A$@$E$-closure of $A$}.
      \item The set
              $\boundary_E A = A^{+E} \smallsetminus A^{-E}$ 
            is called \define{$E$-boundary of $A$}\index{boundary of $A$@$E$-boundary of $A$}.
    \end{enumerate}
  \end{definition}

  \begin{remark}
  \label{rem:group:interior-closure-boundary}
    Let $\mathcal{R}$ be the cell space $\ntuple{\ntuple{G, G, \cdot}, \ntuple{e_G, \family{g}_{g \in G}}}$, where $G$ is a group and $e_G$ is its neutral element. Then, $G_0 = \set{e_G}$ and $\rightsemiaction = \cdot$. Hence, the notions of $E$-interior, $E$-closure, and $E$-boundary are the same as the ones defined in \cite[Sect.~5.4, Paragraph~2]{ceccherini-silberstein:coornaert:2010}.
  \end{remark}


  \begin{example}
  \label{ex:sphere:interior-closure-and-boundary}
    Let $M$ be the Euclidean unit $2$-sphere, that is, the surface of the ball of radius $1$ in $3$-dimensional Euclidean space, and let $G$ be the rotation group. The group $G$ acts transitively but not freely on $M$ on the left by $\leftaction$ by function application, that is, by rotation about the origin. For each point $m \in M$, its orbit is $M$ and its stabiliser is the group of rotations about the line through the origin and itself.

    Furthermore, let $m_0$ be the north pole $(0,0,1)^\transpose$ of $M$ and, for each point $m \in M$, let $g_{m_0, m}$ be a rotation about an axis in the $(x, y)$-plane that rotates $m_0$ to $m$. The stabiliser $G_0$ of the north pole $m_0$ under $\leftaction$ is the group of rotations about the $z$-axis. An element $g G_0 \in G \quotient G_0$ semi-acts on a point $m$ on the right by the induced semi-action $\rightsemiaction$ by first changing the rotation axis of $g$ such that the new axis stands to the line through the origin and $m$ as the old one stood to the line through the origin and $m_0$, $g_{m_0, m} g g_{m_0, m}^{-1}$, and secondly rotating $m$ as prescribed by this new rotation.

    Moreover, let $A$ be a curved circular disk of radius $3 \rho$ with the north pole $m_0$ at its centre, let $g$ be the rotation about an axis $a$ in the $(x,y)$-plane by $\rho$ radians, let $E$ be the set $\set{g_0 g G_0 \suchthat g_0 \in G_0}$, and, for each point $m \in M$, let $E_m$ be the set $m \rightsemiaction E$. Because $G_0$ is the set of rotations about the $z$-axis and $m_0 \rightsemiaction E = g_{m_0, m_0} G_0 g \leftaction m_0 = G_0 \leftaction (g \leftaction m_0)$, the set $E_{m_0}$ is the boundary of a curved circular disk of radius $\rho$ with the north pole $m_0$ at its centre. And, for each point $m \in M$, because $m \rightsemiaction E = g_{m_0, m} \leftaction E_{m_0}$, the set $E_m$ is the boundary of a curved circular disk of radius $\rho$ with $m$ at its centre. 

    The $E$-interior of $A$ is the curved circular disk of radius $2 \rho$ with the north pole $m_0$ at its centre. The $E$-closure of $A$ is the curved circular disk of radius $4 \rho$ with the north pole $m_0$ at its centre. And the $E$-boundary of $A$ is the annulus bounded by the boundaries of the $E$-interior and the $E$-closure of $A$. 
  \end{example}

  Essential properties of and relations between interiors, closures, and boundaries are given in the next lemma. The upper bound given in its corollary follows from the last part of Item~\ref{it:properties-of-interior-closure-and-boundary:finite}.

  \begin{lemma} 
  \label{lem:properties-of-interior-closure-and-boundary}
    Let $A$ be a subset of $M$, let $\family{A_i}_{i \in I}$ be a family of subsets of $M$, let $e$ be an element of $G \quotient G_0$, and let $E$ and $E'$ be two subsets of $G \quotient G_0$.
    \begin{enumerate}
      \item \label{it:properties-of-interior-closure-and-boundary:only-neutral-element}
            $A^{-\set{G_0}} = A$, $A^{+\set{G_0}} = A$, and $\boundary_{\set{G_0}} A = \emptyset$.
      \item \label{it:properties-of-interior-closure-and-boundary:only-neutral-element-and-another} 
            $A^{-\set{G_0, e}} = A \cap (\blank \rightsemiaction e)^{-1}(A)$, $A^{+\set{G_0, e}} = A \cup (\blank \rightsemiaction e)^{-1}(A)$, and $\boundary_{\set{G_0, e}} A = A \smallsetminus (\blank \rightsemiaction e)^{-1}(A) \cup (\blank \rightsemiaction e)^{-1}(A) \smallsetminus A$.
      \item \label{it:properties-of-interior-closure-and-boundary:complement} 
            $(M \smallsetminus A)^{-E} = M \smallsetminus A^{+E}$ and $(M \smallsetminus A)^{+E} = M \smallsetminus A^{-E}$.
      \item \label{it:properties-of-interior-closure-and-boundary:inclusions}
            Let $E \subseteq E'$. Then, $A^{-E} \supseteq A^{-E'}$, $A^{+E} \subseteq A^{+E'}$, and $\boundary_E A \subseteq \boundary_{E'} A$.
      \item \label{it:properties-of-interior-closure-and-boundary:neutral-element}
            Let $G_0 \in E$. Then, $A^{-E} \subseteq A \subseteq A^{+E}$.
      \item Let $G_0 \in E$ and let $A$ be finite. Then, $A^{-E}$ is finite. 
      \item \label{it:properties-of-interior-closure-and-boundary:finite}
            Let $G_0$, $A$, and $E$ be finite. Then, $A^{+E}$ and $\boundary_E A$ are finite. More precisely, $\abs{A^{+E}} \leq \abs{G_0} \cdot \abs{A} \cdot \abs{E}$. 
      \item \label{it:properties-of-interior-closure-and-boundary:commute}
            Let $g \in G$ and let $G_0 \cdot E \subseteq E$. Then, $g \leftaction A^{-E} = (g \leftaction A)^{-E}$, $g \leftaction A^{+E} = (g \leftaction A)^{+E}$, and $g \leftaction \boundary_E A = \boundary_E (g \leftaction A)$.
      \item \label{it:properties-of-interior-closure-and-boundary:commute-with-liberation}
            Let $m \in M$, let $G_0 \cdot E \subseteq E$, and let $\iota \from M \to G \quotient G_0$, $m \mapsto G_{m_0, m}$. Then, $m \rightsemiaction \iota(A^{-E}) = (m \rightsemiaction \iota(A))^{-E}$, $m \rightsemiaction \iota(A^{+E}) = (m \rightsemiaction \iota(A))^{+E}$, and $m \rightsemiaction \iota(\boundary_E A) = \boundary_E (m \rightsemiaction \iota(A))$.
    \end{enumerate}
  \end{lemma}

  \begin{proof}
    \begin{enumerate}
      \item Because $\blank \rightsemiaction G_0 = \identity_M$, this is a direct consequence of Definition~\ref{def:interior-closure-boundary}.
      \item Because $(\blank \rightsemiaction G_0)^{-1}(A) = A$, this is a direct consequence of Definition~\ref{def:interior-closure-boundary}.
      \item For each $m \in M$,
            \begin{align*}
              m \in (M \smallsetminus A)^{-E}
              &\iff m \rightsemiaction E \subseteq M \smallsetminus A\\
              &\iff (m \rightsemiaction E) \cap A = \emptyset\\
              &\iff m \in M \smallsetminus A^{+E}.
            \end{align*}
            Hence, $(M \smallsetminus A)^{-E} = M \smallsetminus A^{+E}$. Therefore,
            \begin{align*}
              (M \smallsetminus A)^{+E}
              &= M \smallsetminus (M \smallsetminus (M \smallsetminus A)^{+E})\\
              &= M \smallsetminus (M \smallsetminus (M \smallsetminus A))^{-E}\\
              &= M \smallsetminus A^{-E}.
            \end{align*}
      \item This is a direct consequence of Definition~\ref{def:interior-closure-boundary}. 
      \item This is a direct consequence of Definition~\ref{def:interior-closure-boundary}.
      \item This is a direct consequence of Item~\ref{it:properties-of-interior-closure-and-boundary:neutral-element}.
      \item Let $e \in E$ and let $a \in A$ such that $(\blank \rightsemiaction e)^{-1}(a) \neq \emptyset$. There are $m$ and $m' \in M$ such that $G_{m_0, m} = e$ and $m' \rightsemiaction e = a$. For each $m'' \in M$, we have $m'' \rightsemiaction e = g_{m_0, m''} \leftaction m$ and hence 
            \begin{align*} 
              m'' \rightsemiaction e = a
              &\iff m'' \rightsemiaction e = m' \rightsemiaction e\\
              &\iff g_{m_0, m'}^{-1} g_{m_0, m''} \leftaction m = m\\
              &\iff g_{m_0, m'}^{-1} g_{m_0, m''} \in G_m\\
              &\iff g_{m_0, m''} \in g_{m_0, m'} G_m.
            \end{align*}
            Moreover, for each $m''$ and each $m''' \in M$ with $m'' \neq m'''$, we have $g_{m_0, m''} \neq g_{m_0, m'''}$. 
            Therefore,
            \begin{align*}
              \abs{(\blank \rightsemiaction e)^{-1}(a)}
              &=    \abs{\set{m'' \in M \suchthat m'' \rightsemiaction e = a}}\\
              &=    \abs{\set{m'' \in M \suchthat g_{m_0, m''} \in g_{m_0, m'} G_m}}\\
              &\leq \abs{g_{m_0, m'} G_m}\\
              &=    \abs{G_m}\\
              &=    \abs{G_0}.
            \end{align*}
            Because $A^{+E} = \bigcup_{e \in E} \bigcup_{a \in A} (\blank \rightsemiaction e)^{-1}(a)$,
            \begin{align*}
              \abs{A^{+E}} &\leq \sum_{e \in E} \sum_{a \in A} \abs{(\blank \rightsemiaction e)^{-1}(a)}\\
                           &\leq \abs{E} \cdot \abs{A} \cdot \abs{G_0}\\
                           &<    \infty.
            \end{align*}
            Because $\boundary_E A \subseteq A^{+E}$, we also have $\abs{\boundary_E A} < \infty$.
      \item Let $m \in M$. Because $\rightsemiaction$ semi-commutes with $\leftaction$, there is a $g_0 \in G_0$ such that $(g^{-1} \leftaction m) \rightsemiaction E = g^{-1} \leftaction (m \rightsemiaction g_0 \cdot E)$. And, because $G_0 \cdot E \subseteq E$, we have $g_0 \cdot E \subseteq E$ and $g_0^{-1} \cdot E \subseteq E$; hence $E = g_0 g_0^{-1} \cdot E = g_0 \cdot (g_0^{-1} \cdot E) \subseteq g_0 \cdot E$; thus $g_0 \cdot E = E$. Therefore, $(g^{-1} \leftaction m) \rightsemiaction E = g^{-1} \leftaction (m \rightsemiaction E)$. 

            Thus, for each $m \in M$,
            \begin{align*}
              m \in g \leftaction A^{-E} &\iff \Exists m' \in A^{-E} \SuchThat g \leftaction m' = m\\
                                         &\iff g^{-1} \leftaction m \in A^{-E}\\
                                         &\iff (g^{-1} \leftaction m) \rightsemiaction E \subseteq A\\
                                         &\iff g^{-1} \leftaction (m \rightsemiaction E) \subseteq A\\
                                         &\iff m \rightsemiaction E \subseteq g \leftaction A\\
                                         &\iff m \in (g \leftaction A)^{-E}.
            \end{align*}
            In conclusion, $g \leftaction A^{-E} = (g \leftaction A)^{-E}$. Moreover, for each $m \in M$,
            \begin{align*}
              m \in g \leftaction A^{+E} &\iff g^{-1} \leftaction m \in A^{+E}\\
                                         &\iff ((g^{-1} \leftaction m) \rightsemiaction E) \cap A \neq \emptyset\\
                                         &\iff (g^{-1} \leftaction (m \rightsemiaction E)) \cap A \neq \emptyset\\
                                         &\iff (m \rightsemiaction E) \cap (g \leftaction A) \neq \emptyset\\
                                         &\iff m \in (g \leftaction A)^{+E}.
            \end{align*}
            In conclusion, $g \leftaction A^{+E} = (g \leftaction A)^{+E}$. Ultimately,
            \begin{align*}
              g \leftaction \boundary_E A &= g \leftaction (A^{+E} \smallsetminus A^{-E})\\
                                          &= (g \leftaction A^{+E}) \smallsetminus (g \leftaction A^{-E})\\
                                          &= (g \leftaction A)^{+E} \smallsetminus (g \leftaction A)^{-E}\\
                                          &= \boundary_E (g \leftaction A).
            \end{align*}
      \item According to
            Item~\ref{it:properties-of-interior-closure-and-boundary:commute},
            \begin{align*}
              m \rightsemiaction \iota(A^{-E})
              &= g_{m_0, m} \leftaction A^{-E}\\
              &= (g_{m_0, m} \leftaction A)^{-E}\\
              &= (m \rightsemiaction \iota(A))^{-E},
            \end{align*}
            and
            \begin{align*}
              m \rightsemiaction \iota(A^{+E})
              &= g_{m_0, m} \leftaction A^{+E}\\
              &= (g_{m_0, m} \leftaction A)^{+E}\\
              &= (m \rightsemiaction \iota(A))^{+E},
            \end{align*}
            and
            \begin{align*}
              m \rightsemiaction \iota(\boundary_E A)
              &= g_{m_0, m} \leftaction \boundary_E A\\
              &= \boundary_E (g_{m_0, m} \leftaction A)\\
              &= \boundary_E (m \rightsemiaction \iota(A)). \tag*{\qed}
            \end{align*}
    \end{enumerate}
  \end{proof}

  \begin{corollary} 
  \label{cor:liberation-preimage}
    Let $G_0$ be finite, let $A$ be a finite subset of $M$, and let $\mathfrak{g}$ be an element of $G \quotient G_0$. Then, $\abs{(\blank \rightsemiaction \mathfrak{g})^{-1}(A)} \leq \abs{G_0} \cdot \abs{A}$.
  \end{corollary}

  \begin{proof}
    This is a direct consequence of Definition~\ref{def:interior-closure-boundary} and Item~\ref{it:properties-of-interior-closure-and-boundary:finite} of Lemma~\ref{lem:properties-of-interior-closure-and-boundary}. \qed
  \end{proof}

  The restriction $\Delta_{X, A}^-$ of $\Delta$ given in Lemma~\ref{lem:Delta-X-A-minus-plus-are-surjective} is well-defined according to the next lemma, which itself holds due to the locality of $\Delta$.

  \begin{lemma} 
  \label{lem:global-transition-function-and-interior-closure}
    Let $\mathcal{C} = \ntuple{\mathcal{R}, Q, N, \delta}$ be a semi-cellular automaton, let $\Delta$ be the global transition function of $\mathcal{C}$, let $c$ and $c'$ be two global configurations of $\mathcal{C}$, and let $A$ be a subset of $M$.
            If $c\restriction_A = c'\restriction_A$, then $\Delta(c)\restriction_{A^{-N}} = \Delta(c')\restriction_{A^{-N}}$.
  \end{lemma}

  \begin{proof}
            Let $c\restriction_A = c'\restriction_A$. Furthermore, let $m \in A^{-N}$. Then, $m \rightsemiaction N \subseteq A$. Hence, $\Delta(c)(m) = \Delta(c')(m)$. \qed
  \end{proof}

  \begin{lemma} 
  \label{lem:Delta-X-A-minus-plus-are-surjective}
    Let $\mathcal{C} = \ntuple{\mathcal{R}, Q, N, \delta}$ be a semi-cellular automaton, let $\Delta$ be the global transition function of $\mathcal{C}$, let $X$ be a subset of $Q^M$, and let $A$ be a subset of $M$. The map
            \begin{align*}
              \Delta_{X, A}^- \from \pi_A(X) &\to     \pi_{A^{-N}}(\Delta(X)),\\
                                           p &\mapsto \Delta(c)\restriction_{A^{-N}}, \text{ where } c \in X \text{ such that } c\restriction_A = p,
            \end{align*}
            is surjective. The map $\Delta_{Q^M, A}^-$ is denoted by $\Delta_A^-$.
%
  \end{lemma}

  \begin{proof}
            Let $p' \in \pi_{A^{-N}}(\Delta(X))$. Then, there is a $c' \in \Delta(X)$ such that $c'\restriction_{A^{-N}} = p'$. Moreover, there is a $c \in X$ such that $\Delta(c) = c'$. Put $p = c\restriction_A \in \pi_A(X)$. Then, $\Delta_{X, A}^-(p) = \Delta(c)\restriction_{A^{-N}} = c'\restriction_{A^{-N}} = p'$. Hence, $\Delta_{X, A}^-$ is surjective. \qed
  \end{proof}

  In the proof of Theorem~\ref{thm:boundary-characterisation-of-folner-net}, the upper bound given in Lemma~\ref{lem:cardinality-of-inverse-image-of-liberation-minus-the-same-less-than-or-equal-to-whatever} is essential, which itself follows from the upper bound given in Corollary~\ref{cor:liberation-preimage} and the inclusion given in Lemma~\ref{lem:liberation-by-n-yields-element-in-set-setminus-bigcup-liberation-by-inverse-times-n-prime}, which in turn follows from the equality given in Lemma~\ref{lem:rightsemiaction-can-be-undone}.

  \begin{lemma}
  \label{lem:rightsemiaction-can-be-undone}
    Let $m$ be an element of $M$, and let $\mathfrak{g}$ be an element of $G \quotient G_0$. There is an element $g \in \mathfrak{g}$ such that
    \begin{equation*}
      \ForEach \mathfrak{g}' \in G \quotient G_0 \Holds (m \rightsemiaction \mathfrak{g}) \rightsemiaction \mathfrak{g}' = m \rightsemiaction g \cdot \mathfrak{g}',
    \end{equation*}
    in particular, for said $g \in \mathfrak{g}$, 
      we have $(m \rightsemiaction \mathfrak{g}) \rightsemiaction g^{-1} G_0 = m$.
  \end{lemma}

  \begin{proof}
    There is a $g \in G$ such, that $g G_0 = \mathfrak{g}$. Moreover, because $\rightsemiaction$ is a semi-action with defect $G_0$, there is a $g_0 \in G_0$ such, that
    \begin{equation*} 
      \ForEach \mathfrak{g}' \in G \quotient G_0 \Holds (m \rightsemiaction g G_0) \rightsemiaction \mathfrak{g}' = m \rightsemiaction g \cdot (g_0^{-1} \cdot \mathfrak{g}').
    \end{equation*}
    Because $g \cdot (g_0^{-1} \cdot \mathfrak{g}') = g g_0^{-1} \cdot \mathfrak{g}'$ and $g g_0^{-1} \in \mathfrak{g}$, the statement holds. \qed
  \end{proof}

  \begin{lemma}
  \label{lem:liberation-by-n-yields-element-in-set-setminus-bigcup-liberation-by-inverse-times-n-prime}
    Let $A$ and $A'$ be two subsets of $M$, and let $\mathfrak{g}$ and $\mathfrak{g}'$ be two elements of $G \quotient G_0$. Then, for each element $m \in (\blank \rightsemiaction \mathfrak{g})^{-1}(A) \smallsetminus (\blank \rightsemiaction \mathfrak{g}')^{-1}(A')$,
    \begin{gather*}
      m \rightsemiaction \mathfrak{g} \in \bigcup_{g \in \mathfrak{g}} A \smallsetminus (\blank \rightsemiaction g^{-1} \cdot \mathfrak{g}')^{-1}(A'),\\
      m \rightsemiaction \mathfrak{g}' \in \bigcup_{g' \in \mathfrak{g}'} (\blank \rightsemiaction (g')^{-1} \cdot \mathfrak{g})^{-1}(A) \smallsetminus A'.
    \end{gather*}
  \end{lemma}

  \begin{proof}
    Let $m \in (\blank \rightsemiaction \mathfrak{g})^{-1}(A) \smallsetminus (\blank \rightsemiaction \mathfrak{g}')^{-1}(A')$. Then, $m \rightsemiaction \mathfrak{g} \in A$ and $m \rightsemiaction \mathfrak{g}' \notin A'$. According to Lemma~\ref{lem:rightsemiaction-can-be-undone}, there is a $g \in \mathfrak{g}$ and a $g' \in \mathfrak{g}'$ such that $(m \rightsemiaction \mathfrak{g}) \rightsemiaction g^{-1} \cdot \mathfrak{g}' = m \rightsemiaction \mathfrak{g}' \notin A'$ and $(m \rightsemiaction \mathfrak{g}') \rightsemiaction (g')^{-1} \cdot \mathfrak{g} = m \rightsemiaction \mathfrak{g} \in A$. Hence, $m \rightsemiaction \mathfrak{g} \notin (\blank \rightsemiaction g^{-1} \cdot \mathfrak{g}')^{-1}(A')$ and $m \rightsemiaction \mathfrak{g}' \in (\blank \rightsemiaction (g')^{-1} \cdot \mathfrak{g})^{-1}(A)$. Therefore, $m \rightsemiaction \mathfrak{g} \in A \smallsetminus (\blank \rightsemiaction g^{-1} \cdot \mathfrak{g}')^{-1}(A')$ and $m \rightsemiaction \mathfrak{g}' \in (\blank \rightsemiaction (g')^{-1} \cdot \mathfrak{g})^{-1}(A) \smallsetminus A'$. In conclusion, $m \rightsemiaction \mathfrak{g} \in \bigcup_{g \in \mathfrak{g}} A \smallsetminus (\blank \rightsemiaction g^{-1} \cdot \mathfrak{g}')^{-1}(A')$ and $m \rightsemiaction \mathfrak{g}' \in \bigcup_{g' \in \mathfrak{g}'} (\blank \rightsemiaction (g')^{-1} \cdot \mathfrak{g})^{-1}(A) \smallsetminus A'$. \qed
  \end{proof}

  \begin{lemma} 
  \label{lem:cardinality-of-inverse-image-of-liberation-minus-the-same-less-than-or-equal-to-whatever}
    Let $G_0$ be finite, let $F$ and $F'$ be two finite subsets of $M$, and let $\mathfrak{g}$ and $\mathfrak{g}'$ be two elements of $G \quotient G_0$. Then,
    \begin{equation*}
      \abs{(\blank \rightsemiaction \mathfrak{g})^{-1}(F) \smallsetminus (\blank \rightsemiaction \mathfrak{g}')^{-1}(F')}
      \leq
      \begin{dcases*} 
        \abs{G_0}^2 \cdot \max_{g \in \mathfrak{g}} \abs{F \smallsetminus (\blank \rightsemiaction g^{-1} \cdot \mathfrak{g}')^{-1}(F')},\\
        \abs{G_0}^2 \cdot \max_{g' \in \mathfrak{g}'} \abs{(\blank \rightsemiaction (g')^{-1} \cdot \mathfrak{g})^{-1}(F) \smallsetminus F'}.
      \end{dcases*}
    \end{equation*} 
  \end{lemma}

  \begin{proof}
    Put $A = (\blank \rightsemiaction \mathfrak{g})^{-1}(F) \smallsetminus (\blank \rightsemiaction \mathfrak{g}')^{-1}(F')$. For each $g \in \mathfrak{g}$, put $B_g = F \smallsetminus (\blank \rightsemiaction g^{-1} \cdot \mathfrak{g}')^{-1}(F')$. For each $g' \in \mathfrak{g}'$, put $B_{g'}' = (\blank \rightsemiaction (g')^{-1} \cdot \mathfrak{g})^{-1}(F) \smallsetminus F'$.

    According to Lemma~\ref{lem:liberation-by-n-yields-element-in-set-setminus-bigcup-liberation-by-inverse-times-n-prime}, the restrictions $(\blank \rightsemiaction \mathfrak{g})\restriction_{A \to \bigcup_{g \in \mathfrak{g}} B_g}$ and $(\blank \rightsemiaction \mathfrak{g}')\restriction_{A \to \bigcup_{g' \in \mathfrak{g}'} B_{g'}'}$ are well-defined. Moreover, for each $m \in M$, according to Corollary~\ref{cor:liberation-preimage}, we have $\abs{(\blank \rightsemiaction \mathfrak{g})^{-1}(m)} \leq \abs{G_0}$ and $\abs{(\blank \rightsemiaction \mathfrak{g}')^{-1}(m)} \leq \abs{G_0}$. Therefore, because $\abs{\mathfrak{g}} = \abs{G_0}$, 
    \begin{equation*}
      \abs{A}
      \leq \abs{G_0} \cdot \abs{\bigcup_{g \in \mathfrak{g}} B_g}
      \leq \abs{G_0} \cdot \sum_{g \in \mathfrak{g}} \abs{B_g}
      \leq \abs{G_0}^2 \cdot \max_{g \in \mathfrak{g}} \abs{B_g}
    \end{equation*}
    and analogously
    \begin{equation*}
      \abs{A} \leq \abs{G_0}^2 \cdot \max_{g' \in \mathfrak{g}'} \abs{B_{g'}'}. \tag*{\qed}
    \end{equation*}
  \end{proof}

%

  \begin{theorem} 
  \label{thm:boundary-characterisation-of-folner-net} 
    Let $G_0$ be finite and let $\net{F_i}_{i \in I}$ be a net in $\set{F \subseteq M \suchthat F \neq \emptyset, F \text{ finite}}$ indexed by $(I, \leq)$. The net $\net{F_i}_{i \in I}$ is a right Følner net in $\mathcal{R}$ if and only if 
    \begin{equation*}
      \ForEach E \subseteq G \quotient G_0 \text{ finite} \Holds \lim_{i \in I} \frac{\abs{\boundary_E F_i}}{\abs{F_i}} = 0.
    \end{equation*}
  \end{theorem}

  \begin{proof}
    First, let $\net{F_i}_{i \in I}$ be a right Følner net in $\mathcal{R}$. Furthermore, let $E \subseteq G \quotient G_0$ be finite. Moreover, let $i \in I$. For each $e \in E$ and each $e' \in E$, put $A_{i,e,e'} = (\blank \rightsemiaction e)^{-1}(F_i) \smallsetminus (\blank \rightsemiaction e')^{-1}(F_i)$. For each $\mathfrak{g} \in G \quotient G_0$, put $B_{i,\mathfrak{g}} = F_i \smallsetminus (\blank \rightsemiaction \mathfrak{g})^{-1}(F_i)$. 
    According to Definition~\ref{def:interior-closure-boundary}, 
    \begin{multline*}
      \boundary_E F_i 
                      = \parens[\Big]{\bigcup_{e \in E} (\blank \rightsemiaction e)^{-1}(F_i)} \smallsetminus \parens[\Big]{\bigcap_{e' \in E} (\blank \rightsemiaction e')^{-1}(F_i)}\\
                      = \bigcup_{e, e' \in E} (\blank \rightsemiaction e)^{-1}(F_i) \smallsetminus (\blank \rightsemiaction e')^{-1}(F_i)
                      = \bigcup_{e, e' \in E} A_{i,e,e'}.
    \end{multline*}
    Hence, $\abs{\boundary_E F_i} \leq \sum_{e, e' \in E} \abs{A_{i,e,e'}}$.

    According to Lemma~\ref{lem:cardinality-of-inverse-image-of-liberation-minus-the-same-less-than-or-equal-to-whatever}, we have $\abs{A_{i,e,e'}} \leq \abs{G_0}^2 \cdot \max_{g \in e} B_{i, g^{-1} \cdot e'}$. Put $E' = \set{g^{-1} \cdot e' \suchthat e, e' \in E, g \in e}$. Because $E$ is finite, $G_0$ is finite, and, for each $e \in E$, we have $\abs{e} = \abs{G_0}$, the set $E'$ is finite. Therefore, 
    \begin{multline*} 
      \frac{\abs{\boundary_E F_i}}{\abs{F_i}}
      \leq \frac{1}{\abs{F_i}} \sum_{e, e' \in E} \abs{A_{i,e,e'}}
      \leq \frac{\abs{G_0}^2}{\abs{F_i}} \sum_{e, e' \in E} \max_{g \in e} \abs{B_{i, g^{-1} \cdot e'}}\\
      \leq \frac{\abs{G_0}^2 \cdot \abs{E}^2}{\abs{F_i}} \max_{e' \in E'} \abs{B_{i, e'}}
      \leq \abs{G_0}^2 \cdot \abs{E}^2 \cdot \max_{e' \in E'} \frac{\abs{F_i \smallsetminus (\blank \rightsemiaction e')^{-1}(F_i)}}{\abs{F_i}} 
      \underset{i \in I}{\to} 0.
    \end{multline*}
    In conclusion, $\lim_{i \in I} \frac{\abs{\boundary_E F_i}}{\abs{F_i}} = 0$.

    Secondly, for each finite $E \subseteq G \quotient G_0$, let $\lim_{i \in I} \frac{\abs{\boundary_E F_i}}{\abs{F_i}} = 0$. Furthermore, let $i \in I$, let $e \in G \quotient G_0$, and put $E = \set{G_0, e}$. According to Item~\ref{it:properties-of-interior-closure-and-boundary:only-neutral-element-and-another} of Lemma~\ref{lem:properties-of-interior-closure-and-boundary}, we have $F_i \smallsetminus (\blank \rightsemiaction e)^{-1}(F_i) \subseteq \boundary_E F_i$. Therefore,
    \begin{equation*}
      \frac{\abs{F_i \smallsetminus (\blank \rightsemiaction e)^{-1}(F_i)}}{\abs{F_i}}
      \leq \frac{\abs{\boundary_E F_i}}{\abs{F_i}}
      \underset{i \in I}{\to} 0.
    \end{equation*}
    In conclusion, $\net{F_i}_{i \in I}$ is a right Følner net in $\mathcal{R}$. \qed
  \end{proof}

  \section{Tilings} 
  \label{sec:tilings}

  In this section, let $\mathcal{R} = \ntuple{\ntuple{M, G, \leftaction}, \ntuple{m_0, \family{g_{m_0, m}}_{m \in M}}}$ be a cell space.

  In Definition~\ref{def:tiling} we introduce the notion of $\ntuple{E,E'}$-tilings. In Theorem~\ref{thm:existence-of-tiling} we show using Zorn's lemma that, for each subset $E$ of $G \quotient G_0$, there is an $\ntuple{E,E'}$-tiling. And in Lemma~\ref{lem:upper-bound-of-tiling-cap-folner-net} we show that, for each $\ntuple{E,E'}$-tiling with finite sets $E$ and $E'$, the net $\net{\abs{T \cap F_i^{-E}}}_{i \in I}$ is asymptotic not less than $\net{\abs{F_i}}_{i \in I}$. 



  \begin{definition} 
  \label{def:tiling}
    Let $T$ be a subset of $M$, and let $E$ and $E'$ be two subsets of $G \quotient G_0$. The set $T$ is called \define{$\ntuple{E, E'}$-tiling of $\mathcal{R}$}\graffito{$\ntuple{E, E'}$-tiling of $\mathcal{R}$}\index{tilingofM@$\ntuple{E, E'}$-tiling of $\mathcal{R}$} if and only if the family $\family{t \rightsemiaction E}_{t \in T}$ is pairwise disjoint and the family $\family{t \rightsemiaction E'}_{t \in T}$ is a cover of $M$. 
  \end{definition}



  \begin{remark} 
  \label{rem:E-and-E-prime-of-tiling-can-be-reduced-or-enlarged}
    Let $T$ be an $\ntuple{E, E'}$-tiling of $\mathcal{R}$. For each subset $F$ of $E$ and each superset $F'$ of $E'$ with $F' \subseteq G \quotient G_0$, the set $T$ is an $\ntuple{F, F'}$-tiling of $\mathcal{R}$. In particular, the set $T$ is an $\ntuple{E, E \cup E'}$-tiling of $\mathcal{R}$.
  \end{remark}


  \begin{remark}
  \label{rem:group:tilings}
    In the situation of Remark~\ref{rem:group:interior-closure-boundary}, the notion of $\ntuple{E, E'}$-tiling is the same as the one defined in \cite[Sect.~5.6, Paragraph~2]{ceccherini-silberstein:coornaert:2010}.
  \end{remark}

  \begin{example}
  \label{ex:sphere:tiling}
    In the situation of Example~\ref{ex:sphere:interior-closure-and-boundary},
    let $E'$ be the set $\set{g (g')^{-1} G_0 \suchthat e, e' \in E, g \in e, g' \in e'}\ (= \set{g_0 g g_0' g^{-1} G_0 \suchthat g_0, g_0' \in G_0})$ and, for each point $m \in M$, let $E'_m = m \rightsemiaction E'$. Because $g^{-1}$ is the rotation about the axis $a$ by $-\rho$ radians, the set $G_0 g^{-1} \leftaction m_0$ is equal to $E_{m_0}$ and the set $g G_0 g^{-1} \leftaction m_0$ is equal to $E_{g \leftaction m_0}$. Because $m_0 \rightsemiaction E' = g_{m_0, m_0} G_0 g G_0 g^{-1} \leftaction m_0 = G_0 \leftaction (g G_0 g^{-1} \leftaction m_0) = G_0 \leftaction E_{g \leftaction m_0}$, the set $E'_{m_0}$ is the curved circular disk of radius $2 \rho$ with the north pole $m_0$ at its centre. And, for each point $m \in M$, because $m \rightsemiaction E' = g_{m_0, m} \leftaction E'_{m_0}$, the set $E'_m$ is the curved circular disk of radius $2 \rho$ with $m$ at its centre.

    If the radius $\rho = \pi / 2$, then the circle $E_{m_0}$ is the equator and the curved circular disk $E'_{m_0}$ has radius $\pi$ and is thus the sphere $M$, and hence the set $T = \set{m_0}$ is an $\ntuple{E, E'}$-tiling of $\mathcal{R}$;
    if the radius $\rho = \pi / 4$, then the curved circular disks $E'_{m_0}$ and $E'_{S}$, where $S$ is the south pole, have radii $\pi / 2$, thus they are hemispheres, and hence the set $T = \set{m_0, S}$ is an $\ntuple{E, E'}$-tiling of $\mathcal{R}$;
    if the radius $\rho = \pi / 8$, then the curved circular disks $E'_{m_0}$ and $E'_{S}$ have radii $\pi / 4$, and it can be shown with spherical geometry that the set $T$ consisting of the north pole $m_0$, the south pole $S$, four equidistant points $m_1$, $m_2$, $m_3$, and $m_4$ on the equator, and the circumcentres $c_1$, $c_2$, $\dotsc$, $c_8$ of the $8$ smallest spherical triangles with one vertex from $\set{m_0, S}$ and two vertices from $\set{m_1, m_2, m_3, m_4}$ (see Fig.~\ref{fig:tiling-of-sphere:pi-eighth}).
    \begin{figure}[bt] 
      \centering
      \includegraphics[trim = 65px 68px 35px 85px, clip]{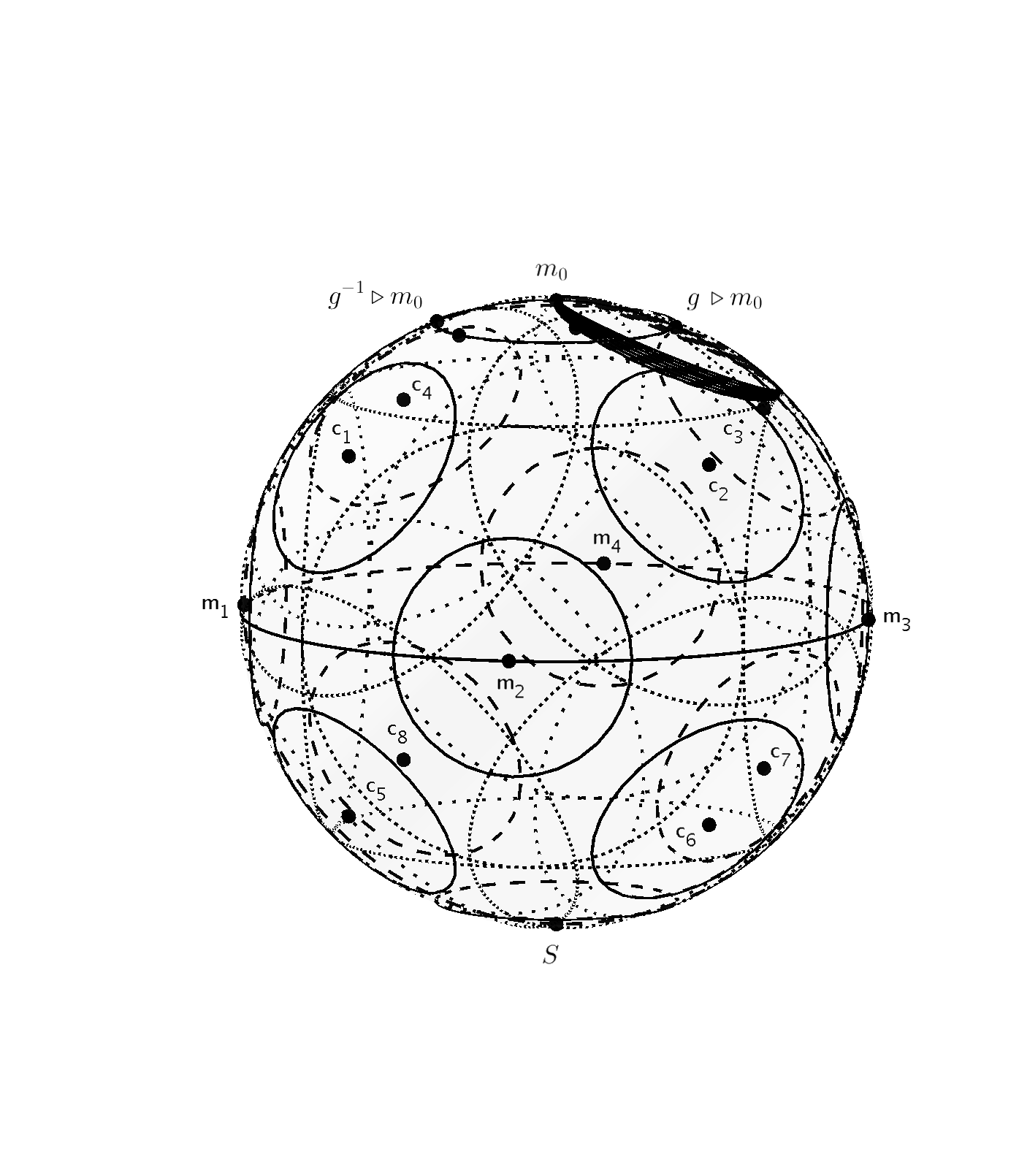} 
      \caption{The points $m_0, S$, $m_1, m_2, m_3, m_4$, $c_1, c_2, \dotsc, c_8$ constitute an $\ntuple{E, E'}$-tiling of the sphere; the circles $E_m$ about these points are drawn solid; the boundaries of the curved circular disks $E'_m$ about these points are drawn dotted; the inclined circle about $g \leftaction m_0$ is the rotation $E_{g \leftaction m_0}$ of $E_{m_0}$ by $\pi / 8$ about the axis $a$; and the other inclined circles are rotations $g_0 \leftaction (E_{g \leftaction m_0})$ of $E_{g \leftaction m_0}$ about the $z$-axis, for a few $g_0 \in G_0$.}
      \label{fig:tiling-of-sphere:pi-eighth}
    \end{figure}
  \end{example}

  \begin{theorem} 
  \label{thm:existence-of-tiling}
    Let $E$ be a non-empty subset of $G \quotient G_0$. There is an $\ntuple{E, E'}$-tiling of $\mathcal{R}$, where 
      $E' = \set{g (g')^{-1} G_0 \suchthat e, e' \in E, g \in e, g' \in e'}$.  
  \end{theorem}

  \begin{proof}
      Let
        $\mathcal{S} = \set{S \subseteq M \suchthat \family{s \rightsemiaction E}_{s \in S} \text{ is pairwise disjoint}}$.
      Because $\set{m_0} \in \mathcal{S}$, the set $\mathcal{S}$ is non-empty. Moreover, it is preordered by inclusion.

      Let $\mathcal{C}$ be a chain in $\ntuple{\mathcal{S}, \subseteq}$. Then, $\bigcup_{S \in \mathcal{C}} S$ is an element of $\mathcal{S}$ and an upper bound of $\mathcal{C}$. According to Zorn's lemma, there is a maximal element $T$ in $\mathcal{S}$. By definition of $\mathcal{S}$, the family $\family{t \rightsemiaction E}_{t \in T}$ is pairwise disjoint. 

%
      Let $m \in M$. Because $T$ is maximal and $m \rightsemiaction E$ is non-empty, there is a $t \in T$ such that $(t \rightsemiaction E) \cap (m \rightsemiaction E) \neq \emptyset$. Hence, there are $e$, $e' \in E$ such that $t \rightsemiaction e = m \rightsemiaction e'$. According to Lemma~\ref{lem:rightsemiaction-can-be-undone}, there is a $g' \in e'$ such that $(m \rightsemiaction e') \rightsemiaction (g')^{-1} G_0 = m$, and there is a $g \in e$ such that $(t \rightsemiaction e) \rightsemiaction (g')^{-1} G_0 = t \rightsemiaction g (g')^{-1} G_0$. Therefore, $m = t \rightsemiaction g (g')^{-1} G_0$. Because $g (g')^{-1} G_0 \in E'$, we have $m \in t \rightsemiaction E'$. Thus, $\family{t \rightsemiaction E'}_{t \in T}$ is a cover of $M$. 

      In conclusion, $T$ is an $\ntuple{E, E'}$-tiling of $\mathcal{R}$. \qed
  \end{proof}


  \begin{lemma} 
  \label{lem:upper-bound-of-tiling-cap-folner-net} 
    Let $G_0$ be finite, let $\net{F_i}_{i \in I}$ be a right Følner net in $\mathcal{R}$ indexed by $(I, \leq)$, let $E$ and $E'$ be two finite subsets of $G \quotient G_0$, and let $T$ be an $\ntuple{E, E'}$-tiling of $\mathcal{R}$. There is a positive real number $\varepsilon \in \R_{> 0}$ and there is an index $i_0 \in I$ such that, for each index $i \in I$ with $i \geq i_0$, we have $\abs{T \cap F_i^{-E}} \geq \varepsilon \abs{F_i}$.
  \end{lemma}

  \begin{proof}
    According to Remark~\ref{rem:E-and-E-prime-of-tiling-can-be-reduced-or-enlarged}, we may suppose, without loss of generality, that $E \subseteq E'$.
    Let $i \in I$. Put
    \begin{equation*}
      T_i^- = T \cap F_i^{-E} = \set{t \in T \suchthat t \rightsemiaction E \subseteq F_i}
    \end{equation*}
    and
    \begin{equation*}
      T_i^+ = T \cap F_i^{+E'} = \set{t \in T \suchthat (t \rightsemiaction E') \cap F_i \neq \emptyset}
    \end{equation*}
    (see Fig.~\ref{fig:upper-bound-of-tiling-cap-folner-net}).
    \begin{figure}[bt]
      \centering
      \begin{minipage}[c]{\textwidth / 2} 
        \begin{tikzpicture}[circle dotted/.style = {dash pattern = on .05mm off 1.5pt, line cap = round}] 
          \pgfmathsetmacro\E{0.4} 
          \pgfmathsetmacro\Ep{0.7} 
          \pgfmathsetmacro\xFi{1 - 0.5} 
          \pgfmathsetmacro\yFi{1 - 0.5} 
          \pgfmathsetmacro\wFi{2.7} 
          \pgfmathsetmacro\hFi{3} 

          \foreach \x in {1, ..., 2} { 
            \foreach \y in {1, ..., 3} {
              \draw (\x, \y) circle (1pt);
            }
          }
          \foreach \x in {0, 3} { 
            \foreach \y in {0, ..., 4} {
              \draw (\x, \y) circle (1pt);
              \fill (\x, \y) circle (0.375pt);
            }
          }
          \foreach \x in {1, 2} { 
            \foreach \y in {0, 4} {
              \draw (\x, \y) circle (1pt);
              \fill (\x, \y) circle (0.375pt);
            }
          }
          \foreach \y in {0, ..., 4} { 
            \fill (4, \y) circle (1pt);
          }
          \foreach \x in {0, ..., 4} {
            \foreach \y in {0, ..., 4} {
              \draw (\x - \E, \y - \E) rectangle (\x + \E, \y + \E); 
              \draw[gray, dashdotted] (\x - \Ep, \y - \Ep) rectangle (\x + \Ep, \y + \Ep); 
            }
          }
          \draw[line width = 0.5pt, dashed] (\xFi, \yFi) rectangle (\xFi + \wFi, \yFi + \hFi); 
          \draw[line width = 0.75pt, circle dotted] (\xFi + \E, \yFi + \E) rectangle (\xFi + \wFi - \E, \yFi + \hFi - \E); 
          \draw[line width = 0.75pt, circle dotted] (\xFi - \Ep, \yFi - \Ep) rectangle (\xFi + \wFi + \Ep, \yFi + \hFi + \Ep); 
          \draw[line width = 0.75pt, circle dotted] (\xFi + \Ep, \yFi + \Ep) rectangle (\xFi + \wFi - \Ep, \yFi + \hFi - \Ep); 
        \end{tikzpicture}
      \end{minipage}%
      \begin{minipage}[c]{\textwidth / 2}
        The whole space is $M$; the dots, circles, and circles with dots are the elements of the tiling $T$; for each element $t \in T$, the region enclosed by the rectangle with solid border centred at $t$ is the set $t \rightsemiaction E$ and the region enclosed by the rectangle with dash-dotted border centred at $t$ is the set $t \rightsemiaction E'$; the region enclosed by the rectangle with dashed border is $F_i$; the region enclosed by the smallest rectangle with dotted border is $F_i^{-E'}$, the region enclosed by the second smallest rectangle with dotted border is $F_i^{-E}$, and the region enclosed by the largest rectangle with dotted border is $F_i^{+E'}$; the circles are the elements of $T_i^- = T \cap F_i^{-E}$, and the circles with and without dots are the elements of $T_i^+ = T \cap F_i^{+E'}$. 
      \end{minipage}
      \caption{Schematic representation of the set-up of the proof of Lemma~\ref{lem:upper-bound-of-tiling-cap-folner-net}.}
      \label{fig:upper-bound-of-tiling-cap-folner-net}
    \end{figure}
    Because $T$ is an $\ntuple{E, E'}$-tiling of $\mathcal{R}$,
    \begin{equation*}
      F_i = M \cap F_i
          = \parens[\big]{\bigcup_{t \in T} t \rightsemiaction E'} \cap F_i
          = \bigcup_{t \in T} (t \rightsemiaction E') \cap F_i.
    \end{equation*}
    Moreover, for each $t \in T \smallsetminus T_i^+$, we have $(t \rightsemiaction E') \cap F_i = \emptyset$. Hence,
    \begin{equation*}
      F_i = \bigcup_{t \in T_i^+} (t \rightsemiaction E') \cap F_i.
    \end{equation*}
    Therefore, $F_i \subseteq \bigcup_{t \in T_i^+} t \rightsemiaction E'$. Because $\rightsemiaction$ is free, for each $t \in T_i^+$, we have $\abs{t \rightsemiaction E'} = \abs{E'}$. Hence, $\abs{F_i} \leq \abs{T_i^+} \cdot \abs{E'}$. Thus, because $E' \neq \emptyset$, 
    \begin{equation*}
      \frac{\abs{T_i^+}}{\abs{F_i}} \geq \frac{1}{\abs{E'}}.
    \end{equation*}
    Because $E \subseteq E'$, according to Item~\ref{it:properties-of-interior-closure-and-boundary:inclusions} of Lemma~\ref{lem:properties-of-interior-closure-and-boundary}, we have $F_i^{-E} \supseteq F_i^{-E'}$. Therefore,
    \begin{equation*}
      T_i^+ \smallsetminus T_i^-
      =         T \cap (F_i^{+E'} \smallsetminus F_i^{-E})
      \subseteq T \cap (F_i^{+E'} \smallsetminus F_i^{-E'})
      =         T \cap \boundary_{E'} F_i
      \subseteq \boundary_{E'} F_i.
    \end{equation*}
    Hence, $\abs{\boundary_{E'} F_i} \geq \abs{T_i^+ \smallsetminus T_i^-} \geq \abs{T_i^+} - \abs{T_i^-}$. Therefore, 
    \begin{equation*}
      \frac{\abs{T_i^-}}{\abs{F_i}}
      \geq \frac{\abs{T_i^+}}{\abs{F_i}} - \frac{\abs{\boundary_{E'} F_i}}{\abs{F_i}}
      \geq \frac{1}{\abs{E'}} - \frac{\abs{\boundary_{E'} F_i}}{\abs{F_i}}.
    \end{equation*}
    Moreover, according to Theorem~\ref{thm:boundary-characterisation-of-folner-net}, there is an $i_0 \in I$ such that 
    \begin{equation*}
      \ForEach i \in I \Holds i \geq i_0 \implies \frac{\abs{\boundary_{E'} F_i}}{\abs{F_i}} \leq \frac{1}{2 \abs{E'}}.
    \end{equation*}
    Put $\varepsilon = 1 / (2 \abs{E'})$. Then, for each $i \in I$ with $i \geq i_0$,
    \begin{equation*}
      \frac{\abs{T_i^-}}{\abs{F_i}} \geq \frac{1}{2 \abs{E'}} = \varepsilon. \tag*{\qed}
    \end{equation*}
  \end{proof}

  \section{Entropies} 
  \label{sec:entropies}

  In this section, let $\mathcal{R} = \ntuple{\ntuple{M, G, \leftaction}, \ntuple{m_0, \family{g_{m_0, m}}_{m \in M}}}$ be a right amenable cell space, let $\mathcal{C} = \ntuple{\mathcal{R}, Q, N, \delta}$ be a semi-cellular automaton, and let $\Delta$ be the global transition function of $\mathcal{C}$ such that the stabiliser $G_0$ of $m_0$ under $\leftaction$, the set $Q$ of states, and the neighbourhood $N$ are finite, and the set $Q$ is non-empty. 

  In Definition~\ref{def:entropy} we introduce the entropy of a subset $X$ of $Q^M$ with respect to a net $\net{F_i}_{i \in I}$ of non-empty and finite subsets of $M$, which is the asymptotic growth rate of the number of finite patterns with domain $F_i$ that occur in $X$. In Lemma~\ref{lem:entropy-basic-facts} we show that $Q^M$ has entropy $\log\abs{Q}$ and that entropy is non-decreasing. In Theorem~\ref{thm:entropy-does-non-increase} we show that applications of global transition functions of cellular automata on subsets of $Q^M$ do not increase their entropy. And in Lemma~\ref{lem:entorpy-bounded-above-if-strange-tiling-exists} we show that if for each point $t$ of an $\ntuple{E,E'}$-tiling not all patterns with domain $t \rightsemiaction E$ occur in a subset of $Q^M$, then that subset has less entropy than $Q^M$. 

  \begin{definition}
  \label{def:entropy}
    Let $X$ be a subset of $Q^M$ and let $\mathcal{F} = \net{F_i}_{i \in I}$ be a net in $\set{F \subseteq M \suchthat F \neq \emptyset, F \text{ finite}}$. The non-negative real number 
    \begin{equation*}
      \entropy_{\mathcal{F}}(X) = \limsup\limits_{i \in I} \frac{\log\abs{\pi_{F_i}(X)}}{\abs{F_i}}
    \end{equation*}
    is called \define{entropy of $X$ with respect to $\mathcal{F}$}.
  \end{definition}

  \begin{remark}
  \label{rem:group:entropy}
    In the situation of Remark~\ref{rem:group:interior-closure-boundary}, the notion of entropy is the same as the one defined in \cite[Definition~5.7.1]{ceccherini-silberstein:coornaert:2010}.
  \end{remark}

  \begin{lemma}
  \label{lem:entropy-basic-facts}
    Let $\mathcal{F} = \net{F_i}_{i \in I}$ be a net in $\set{F \subseteq M \suchthat F \neq \emptyset, F \text{ finite}}$. Then,
    \begin{enumerate}
      \item \label{it:entropy-basic-facts:whole-space}
            $\entropy_{\mathcal{F}}(Q^M) = \log\abs{Q}$;
      \item \label{it:entropy-basic-facts:monotonic}
            $\ForEach X \subseteq Q^M \ForEach X' \subseteq Q^M \Holds \parens[\big]{X \subseteq X' \implies \entropy_{\mathcal{F}}(X) \leq \entropy_{\mathcal{F}}(X')}$;
      \item \label{it:entropy-basic-facts:bound}
            $\ForEach X \subseteq Q^M \Holds \entropy_{\mathcal{F}}(X) \leq \log\abs{Q}$.
    \end{enumerate}
  \end{lemma}

  \begin{proof}
    \begin{enumerate}
      \item For each $i \in I$, we have $\pi_{F_i}(Q^M) = Q^{F_i}$ and hence
            \begin{equation*}
              \frac{\log\abs{\pi_{F_i}(Q^M)}}{\abs{F_i}} 
                                                         = \frac{\log\abs{Q}^{\abs{F_i}}}{\abs{F_i}}
                                                         = \frac{\abs{F_i} \cdot \log\abs{Q}}{\abs{F_i}}
                                                         = \log\abs{Q}.
            \end{equation*}
            In conclusion, $\entropy_{\mathcal{F}}(Q^M) = \log\abs{Q}$.
      \item Let $X$, $X' \subseteq Q^M$ such that $X \subseteq X'$. For each $i \in I$, we have $\pi_{F_i}(X) \subseteq \pi_{F_i}(X')$ and hence, because $\log$ is non-decreasing, $\log\abs{\pi_{F_i}(X)} \leq \log\abs{\pi_{F_i}(X')}$. In conclusion, $\entropy_{\mathcal{F}}(X) \leq \entropy_{\mathcal{F}}(X')$. 
      \item This is a direct consequence of Items~\ref{it:entropy-basic-facts:monotonic} and~\ref{it:entropy-basic-facts:whole-space}. \qed
    \end{enumerate}
  \end{proof}

  In the remainder of this section, let $\mathcal{F} = \net{F_i}_{i \in I}$ be a right Følner net in $\mathcal{R}$ indexed by $(I, \leq)$. 

  \begin{theorem}
  \label{thm:entropy-does-non-increase}
    Let $X$ be a subset of $Q^M$. Then, $\entropy_{\mathcal{F}}(\Delta(X)) \leq \entropy_{\mathcal{F}}(X)$.
  \end{theorem} 

  \begin{proof}
    Suppose, without loss of generality, that $G_0 \in N$.
    Let $i \in I$. According to Lemma~\ref{lem:Delta-X-A-minus-plus-are-surjective}, the map
    $\Delta_{X, F_i}^- \from \pi_{F_i}(X) \to \pi_{F_i^{-N}}(\Delta(X))$
    is surjective. Therefore, $\abs{\pi_{F_i^{-N}}(\Delta(X))} \leq \abs{\pi_{F_i}(X)}$.
    Because $G_0 \in N$, according to Item~\ref{it:properties-of-interior-closure-and-boundary:neutral-element} of Lemma~\ref{lem:properties-of-interior-closure-and-boundary}, we have $F_i^{-N} \subseteq F_i$. Thus, $\pi_{F_i}(\Delta(X)) \subseteq \pi_{F_i^{-N}}(\Delta(X)) \times Q^{F_i \smallsetminus F_i^{-N}}$. Hence, 
    \begin{align*} 
      \log\abs{\pi_{F_i}(\Delta(X))}
      &\leq \log\abs{\pi_{F_i^{-N}}(\Delta(X))} + \log\abs{Q^{F_i \smallsetminus F_i^{-N}}}\\ 
      &\leq 
             \log\abs{\pi_{F_i}(X)} + \abs{F_i \smallsetminus F_i^{-N}} \cdot \log\abs{Q}.
    \end{align*}
    Because $G_0 \in N$, according to Item~\ref{it:properties-of-interior-closure-and-boundary:neutral-element} of Lemma~\ref{lem:properties-of-interior-closure-and-boundary}, we have $F_i \subseteq F_i^{+N}$. Therefore, $F_i \smallsetminus F_i^{-N} \subseteq F_i^{+N} \smallsetminus F_i^{-N} = \boundary_N F_i$. Because $G_0$, $F_i$, and $N$ are finite, according to Item~\ref{it:properties-of-interior-closure-and-boundary:finite} of Lemma~\ref{lem:properties-of-interior-closure-and-boundary}, the boundary $\boundary_N F_i$ is finite. Hence,
    \begin{equation*}
      \frac{\log\abs{\pi_{F_i}(\Delta(X))}}{\abs{F_i}} \leq \frac{\log\abs{\pi_{F_i}(X)}}{\abs{F_i}} + \frac{\abs{\boundary_N F_i}}{\abs{F_i}} \log\abs{Q}.
    \end{equation*}
    Therefore, because $N$ is finite, according to Theorem~\ref{thm:boundary-characterisation-of-folner-net},
    \begin{equation*}
      \entropy_{\mathcal{F}}(\Delta(X))
      \leq \limsup\limits_{i \in I} \frac{\log\abs{\pi_{F_i}(X)}}{\abs{F_i}} + \parens*{\lim_{i \in I} \frac{\abs{\boundary_N F_i}}{\abs{F_i}}} \cdot \log\abs{Q}
      =    \entropy_{\mathcal{F}}(X). \tag*{\qed}
    \end{equation*}
  \end{proof}

  \begin{lemma} 
  \label{lem:entorpy-bounded-above-if-strange-tiling-exists}
    Let $Q$ contain at least two elements, let $X$ be a subset of $Q^M$, let $E$ and $E'$ be two non-empty and finite subsets of $G \quotient G_0$, and let $T$ be an $\ntuple{E, E'}$-tiling of $\mathcal{R}$, such that, for each cell $t \in T$, we have $\pi_{t \rightsemiaction E}(X) \subsetneqq Q^{t \rightsemiaction E}$. Then, $\entropy_{\mathcal{F}}(X) < \log\abs{Q}$.
  \end{lemma}

  \begin{proof}
    For each $t \in T$, because $\pi_{t \rightsemiaction E}(X) \subsetneqq Q^{t \rightsemiaction E}$, $\abs{Q} \geq 2$, and $\abs{t \rightsemiaction E} \geq 1$,
    \begin{equation*}
      \abs{\pi_{t \rightsemiaction E}(X)}
      \leq \abs{Q^{t \rightsemiaction E}} - 1
      =    \abs{Q}^{\abs{t \rightsemiaction E}} - 1
      \geq 1.
    \end{equation*}
    Let $i \in I$. Put $T_i = T \cap F_i^{-E}$ and put $F_i^* = F_i \smallsetminus (\bigcup_{t \in T_i} t \rightsemiaction E)$ (see Fig.~\ref{fig:entorpy-bounded-above-if-strange-tiling-exists}).
    \begin{figure}[bt]
      \centering
      \begin{minipage}[c]{\textwidth / 2} 
        \begin{tikzpicture}[circle dotted/.style = {dash pattern = on .05mm off 1.5pt, line cap = round}] 
          \pgfmathsetmacro\E{0.4} 
          \pgfmathsetmacro\Ep{0.7} 
          \pgfmathsetmacro\xFi{1 - 0.5} 
          \pgfmathsetmacro\yFi{1 - 0.5} 
          \pgfmathsetmacro\wFi{2.7} 
          \pgfmathsetmacro\hFi{3} 

          \foreach \x in {0, ..., 4} {
            \foreach \y in {0, ..., 4} {
              \path[fill] (\x, \y) circle (1pt); 
              \draw (\x - \E, \y - \E) rectangle (\x + \E, \y + \E); 
            }
          }
          \draw[dashed, pattern = north east lines] (\xFi, \yFi) rectangle (\xFi + \wFi, \yFi + \hFi); 
          \foreach \x in {1, ..., 2} {
            \foreach \y in {1, ..., 3} {
              \draw[fill = white] (\x - \E, \y - \E) rectangle (\x + \E, \y + \E); 
              \draw (\x, \y) circle (1pt); 
            }
          }
          \draw[line width = 0.75pt, circle dotted] (\xFi + \E, \yFi + \E) rectangle (\xFi + \wFi - \E, \yFi + \hFi - \E); 
        \end{tikzpicture}
      \end{minipage}%
      \begin{minipage}[c]{\textwidth / 2}
        The whole space is $M$; the dots and circles are the elements of the tiling $T$; for each element $t \in T$, the region enclosed by the rectangle with solid border centred at $t$ is the set $t \rightsemiaction E$; the region enclosed by the rectangle with dashed border is $F_i$; the region enclosed by the rectangle with dotted border is $F_i^{-E}$; the circles are the elements of $T_i = T \cap F_i^{-E}$; the hatched region is $F_i^* = F_i \smallsetminus (\bigcup_{t \in T_i} t \rightsemiaction E)$. 
      \end{minipage}
      \caption{Schematic representation of the set-up of the proof of Lemma~\ref{lem:entorpy-bounded-above-if-strange-tiling-exists}.}
      \label{fig:entorpy-bounded-above-if-strange-tiling-exists}
    \end{figure}
    Because $\bigcup_{t \in T_i} t \rightsemiaction E \subseteq F_i$ and $\family{t \rightsemiaction E}_{t \in T}$ is pairwise disjoint,
    \begin{equation*}
      \pi_{F_i}(X) \subseteq \pi_{F_i^*}(X) \times \prod_{t \in T_i} \pi_{t \rightsemiaction E}(X)
                   \subseteq Q^{F_i^*} \times \prod_{t \in T_i} \pi_{t \rightsemiaction E}(X).
    \end{equation*}
    Therefore,
    \begin{align*}
      \log\abs{\pi_{F_i}(X)}
      &\leq \log\abs{Q}^{\abs{F_i^*}} + \sum_{t \in T_i} \log\abs{\pi_{t \rightsemiaction E}(X)}\\ 
      &\leq \log\abs{Q}^{\abs{F_i^*}} + \sum_{t \in T_i} \log\parens[\big]{\abs{Q}^{\abs{t \rightsemiaction E}} - 1}\\ 
      &=    \abs{F_i^*} \cdot \log\abs{Q} + \sum_{t \in T_i} \log\parens[\big]{\abs{Q}^{\abs{t \rightsemiaction E}} (1 - \abs{Q}^{- \abs{t \rightsemiaction E}})}\\
      &=    \abs{F_i^*} \cdot \log\abs{Q} + \sum_{t \in T_i} \abs{t \rightsemiaction E} \cdot \log\abs{Q} + \sum_{t \in T_i} \log\parens[\big]{1 - \abs{Q}^{- \abs{t \rightsemiaction E}}}.
    \end{align*}
    Moreover, for each $t \in T_i$, we have $t \rightsemiaction E \subseteq F_i$. Thus,
    \begin{equation*}
      \abs{F_i^*} = \abs{F_i} - \sum_{t \in T_i} \abs{t \rightsemiaction E}.
    \end{equation*}
    And, because $\rightsemiaction$ is free, we have $\abs{t \rightsemiaction E} = \abs{E}$. Hence,
    \begin{equation*}
      \log\abs{\pi_{F_i}(X)}
      \leq    \abs{F_i} \cdot \log\abs{Q} + \abs{T_i} \cdot \log\parens[\big]{1 - \abs{Q}^{- \abs{E}}}. 
    \end{equation*}
    Put $c = - \log\parens[\big]{1 - \abs{Q}^{- \abs{E}}}$. Because $\abs{Q} \geq 2$ and $\abs{E} \geq 1$, we have $\abs{Q}^{- \abs{E}} \in (0,1)$ and hence $c > 0$. 
    According to Lemma~\ref{lem:upper-bound-of-tiling-cap-folner-net}, there are $\varepsilon \in \R_{> 0}$ and $i_0 \in I$ such that, for each $i \in I$ with $i \geq i_0$, we have $\abs{T_i} \geq \varepsilon \abs{F_i}$. Therefore, for each such $i$, 
    \begin{equation*}
      \frac{\log\abs{\pi_{F_i}(X)}}{\abs{F_i}} \leq \log\abs{Q} - c \varepsilon.
    \end{equation*}
    In conclusion,
    \begin{equation*}
      \entropy_{\mathcal{F}}(X)
      =    \limsup_{i \in I} \frac{\log\abs{\pi_{F_i}(X)}}{\abs{F_i}}\\
      \leq \log\abs{Q} - c \varepsilon\\
      <    \log\abs{Q}. \tag*{\qed}
    \end{equation*}
  \end{proof}

  \begin{corollary} 
  \label{cor:nice-properties-yield-less-entropy}
    Let $Q$ contain at least two elements, let $X$ be a $\inducedleftaction$-invariant subset of $Q^M$, and let $E$ be a non-empty and finite subset of $G \quotient G_0$, such that $\pi_{m_0 \rightsemiaction E}(X) \subsetneqq Q^{m_0 \rightsemiaction E}$. Then, $\entropy_{\mathcal{F}}(X) < \log\abs{Q}$.
  \end{corollary}

  \begin{proof}
    According to Theorem~\ref{thm:existence-of-tiling}, there is a subset $E'$ of $G \quotient G_0$ and an $\ntuple{E, E'}$-tiling $T$ of $\mathcal{R}$. Because $G_0$ and $E$ are finite, so is $E'$. Let $m \in M$.
    Put $g = g_{m_0, m_0} g_{m_0, m}^{-1}$. Then, $g \leftaction (m \rightsemiaction E) = m_0 \rightsemiaction E$.
    Because $X$ is $\inducedleftaction$-invariant, 
    \begin{align*}
      \pi_{m \rightsemiaction E}(X)
      = \pi_{m \rightsemiaction E}(g^{-1} \inducedleftaction X)
      = g^{-1} \inducedleftaction \pi_{g \leftaction (m \rightsemiaction E)}(X)
      = g^{-1} \inducedleftaction \pi_{m_0 \rightsemiaction E}(X).
    \end{align*}
    Because $\pi_{m_0 \rightsemiaction E}(X) \subsetneqq Q^{m_0 \rightsemiaction E}$, 
    \begin{align*}
      g^{-1} \inducedleftaction \pi_{m_0 \rightsemiaction E}(X)
      \subsetneqq g^{-1} \inducedleftaction Q^{m_0 \rightsemiaction E}
      =           Q^{g^{-1} \leftaction (m_0 \rightsemiaction E)}
      =           Q^{m \rightsemiaction E}.
    \end{align*}
    Therefore, $\pi_{m \rightsemiaction E}(X) \subsetneqq Q^{m \rightsemiaction E}$. In conclusion, according to Lemma~\ref{lem:entorpy-bounded-above-if-strange-tiling-exists}, we have $\entropy_{\mathcal{F}}(X) < \log\abs{Q}$. \qed
  \end{proof}

  \section{Gardens of Eden} 
  \label{sec:gardens-of-eden}

  In this section, let $\mathcal{R} = \ntuple{\ntuple{M, G, \leftaction}, \ntuple{m_0, \family{g_{m_0, m}}_{m \in M}}}$ be a right amenable cell space and let $\mathcal{C} = \ntuple{\mathcal{R}, Q, N, \delta}$ be a semi-cellular automaton such that the stabiliser $G_0$ of $m_0$ under $\leftaction$, the set $Q$ of states, and the neighbourhood $N$ are finite, and the set $Q$ is non-empty. Furthermore, let $\Delta$ be the global transition function of $\mathcal{C}$, and let $\mathcal{F} = \net{F_i}_{i \in I}$ be a right Følner net in $\mathcal{R}$ indexed by $(I, \leq)$.

  In Theorem~\ref{thm:not-surjective-implies-less-entropy} we show that if $\Delta$ is not surjective, then the entropy of its image is less than the entropy of $Q^M$. And the converse of that statement obviously holds. In Theorem~\ref{thm:less-entropy-implies-not-pre-injective} we show that if the entropy of the image of $\Delta$ is less than the entropy of $Q^M$, then $\Delta$ is not pre-injective. And in Theorem~\ref{thm:not-pre-injective-implies-less-entropy} we show the converse of that statement. These four statements establish the Garden of Eden theorem, see Main Theorem~\ref{thm:garden-of-eden}. 




  \begin{definition} 
    Let $c$ and $c'$ be two maps from $M$ to $Q$. The set
      $\diff(c, c') = \set{m \in M \suchthat c(m) \neq c'(m)}$ 
    is called \define{difference of $c$ and $c'$}.
  \end{definition}

  \begin{definition} 
    The map $\Delta$ is called \define{pre-injective}\graffito{pre-injective} if and only if, for each tuple $(c, c') \in Q^M \times Q^M$ such that $\diff(c, c')$ is finite and $\Delta(c) = \Delta(c')$, we have $c = c'$.
  \end{definition}

  In the proof of Theorem~\ref{thm:not-surjective-implies-less-entropy}, the existence of a Garden of Eden pattern, as stated in Lemma~\ref{lem:not-surjective-yields-garden-of-eden-pattern}, is essential, which itself follows from the existence of a Garden of Eden configuration, the compactness of $Q^M$, and the continuity of $\Delta$.

  \begin{definition}
    \begin{enumerate}
      \item Let $c \from M \to Q$ be a global configuration. It is called \define{Garden of Eden configuration}\graffito{Garden of Eden configuration $c$ of $\mathcal{C}$} if and only if it is not contained in $\Delta(Q^M)$. 
      \item Let $p \from A \to Q$ be a pattern. It is called \define{Garden of Eden pattern}\graffito{Garden of Eden pattern $p$ of $\mathcal{C}$} if and only if, for each global configuration $c \in Q^M$, we have $\Delta(c)\restriction_A \neq p$.  
    \end{enumerate}
  \end{definition}

  \begin{remark}
    \begin{enumerate}
      \item The global transition function $\Delta$ is surjective if and only if there is no Garden of Eden configuration. 
      \item If $p \from A \to Q$ is a Garden of Eden pattern, then each global configuration $c \in Q^M$ with $c\restriction_A = p$ is a Garden of Eden configuration. 
      \item If there is a Garden of Eden pattern, then $\Delta$ is not surjective. 
    \end{enumerate}
  \end{remark}

  \begin{lemma} 
  \label{lem:not-surjective-yields-garden-of-eden-pattern}
    Let $\Delta$ not be surjective. There is a Garden of Eden pattern with non-empty and finite domain. 
  \end{lemma}

  \begin{proof}
    Because $\Delta$ is not surjective, there is a Garden of Eden configuration $c \in Q^M$. Equip $Q^M$ with the prodiscrete topology. According to
    \cite[Lemma~3.3.2]{ceccherini-silberstein:coornaert:2010},
    $\Delta(Q^M)$ is closed in $Q^M$. Hence, $Q^M \smallsetminus \Delta(Q^M)$ is open. Therefore, because $c \in Q^M \smallsetminus \Delta(Q^M)$, there is a non-empty and finite subset $F$ of $M$ such that 
    \begin{equation*} 
      \Cyl(c, F) = \set{c' \in Q^M \suchthat c'\restriction_F = c\restriction_F} \subseteq Q^M \smallsetminus \Delta(Q^M).
    \end{equation*}
    Hence, $c\restriction_F$ is a Garden of Eden pattern with non-empty and finite domain. \qed
  \end{proof}


  \begin{theorem} 
  \label{thm:not-surjective-implies-less-entropy}
    Let $\delta$ be $\bullet$-invariant, let $Q$ contain at least two elements, and let $\Delta$ not be surjective. Then, $\entropy_{\mathcal{F}}(\Delta(Q^M)) < \log\abs{Q}$.
  \end{theorem}

  \begin{proof}
    According to Lemma~\ref{lem:not-surjective-yields-garden-of-eden-pattern}, there is a Garden of Eden pattern $p \from F \to Q$ with non-empty and finite domain. Let $E = (m_0 \rightsemiaction \blank)^{-1}(F)$. Then, $m_0 \rightsemiaction E = F$ and, because $\rightsemiaction$ is free, $\abs{E} = \abs{F} < \infty$. Because $p$ is a Garden of Eden pattern, $p \notin \pi_{m_0 \rightsemiaction E}(\Delta(Q^M))$. Hence, $\pi_{m_0 \rightsemiaction E}(\Delta(Q^M)) \subsetneqq Q^{m_0 \rightsemiaction E}$. Moreover, according to
    \cite[Item~1 of Theorem~2]{wacker:automata:2016},
    the map $\Delta$ is $\inducedleftaction$-equivariant. Hence, for each $g \in G$, we have $g \inducedleftaction \Delta(Q^M) = \Delta(g \inducedleftaction Q^M) = \Delta(Q^M)$. In other words, $\Delta(Q^M)$ is $\inducedleftaction$-invariant. Thus, according to Corollary~\ref{cor:nice-properties-yield-less-entropy}, we have $\entropy_{\mathcal{F}}(\Delta(Q^M)) < \log\abs{Q}$. \qed
  \end{proof}

  In the proof of Theorem~\ref{thm:less-entropy-implies-not-pre-injective}, the fact that enlarging each element of $\mathcal{F}$ does not increase entropy, as stated in the next lemma, is essential.

  \begin{lemma}
  \label{lem:entropy-invariant-under-closure-net-change}
    Let $X$ be a subset of $Q^M$ and let $E$ be a finite subset of $G \quotient G_0$ such that $G_0 \in E$. Then,
      $\entropy_{\net{F_i^{+E}}_{i \in I}}(X) \leq \entropy_{\mathcal{F}}(X)$.  
  \end{lemma}

  \begin{proof}
    Let $i \in I$. According to Item~\ref{it:properties-of-interior-closure-and-boundary:neutral-element} of Lemma~\ref{lem:properties-of-interior-closure-and-boundary}, we have $F_i^{-E} \subseteq F_i \subseteq F_i^{+E}$. Hence, $\pi_{F_i^{+E}}(X) \subseteq \pi_{F_i}(X) \times Q^{F_i^{+E} \smallsetminus F_i}$ and $F_i^{+E} \smallsetminus F_i \subseteq \boundary_E F_i$. Thus,
    \begin{equation*}
      \log\abs{\pi_{F_i^{+E}}(X)}
      \leq \log\abs{\pi_{F_i}(X)} + \abs{F_i^{+E} \smallsetminus F_i} \cdot \log\abs{Q} 
      \leq \log\abs{\pi_{F_i}(X)} + \abs{\boundary_E F_i} \cdot \log\abs{Q}.
    \end{equation*}
    Therefore, according to Theorem~\ref{thm:boundary-characterisation-of-folner-net},
    \begin{equation*}
      \entropy_{\net{F_i^{+E}}_{i \in I}}(X)
      \leq \limsup_{i \in I} \frac{\log\abs{\pi_{F_i}(X)}}{\abs{F_i}} + \parens*{\lim_{i \in I} \frac{\abs{\boundary_E F_i}}{\abs{F_i}}} \cdot \log\abs{Q}\\
      =    \entropy_{\mathcal{F}}(X). \tag*{\qed}
    \end{equation*}
  \end{proof}

  \begin{theorem} 
  \label{thm:less-entropy-implies-not-pre-injective}
    Let $\entropy_{\mathcal{F}}(\Delta(Q^M)) < \log\abs{Q}$. Then, $\Delta$ is not pre-injective.
  \end{theorem}

  \begin{proof}
    Suppose, without loss of generality, that $G_0 \in N$. Let $X = \Delta(Q^M)$. According to Lemma~\ref{lem:entropy-invariant-under-closure-net-change}, we have $\entropy_{\net{F_i^{+N}}_{i \in I}}(X) \leq \entropy_{\mathcal{F}}(X) < \log\abs{Q}$. Hence, there is an $i \in I$ such that
    \begin{equation*}
      \frac{\log\abs{\pi_{F_i^{+N}}(X)}}{\abs{F_i}} < \log\abs{Q}.
    \end{equation*}
    Thus, $\abs{\pi_{F_i^{+N}}(X)} < \abs{Q}^{\abs{F_i}}$. Furthermore, let $q \in Q$ and let $X' = \set{c \in Q^M \suchthat c\restriction_{M \smallsetminus F_i} \equiv q}$. Then, $\abs{Q}^{\abs{F_i}} = \abs{X'}$. Hence,
      $\abs{\pi_{F_i^{+N}}(X)} < \abs{X'}$.
    Moreover, according to Item~\ref{it:properties-of-interior-closure-and-boundary:complement} of Lemma~\ref{lem:properties-of-interior-closure-and-boundary}, we have $(M \smallsetminus F_i)^{-N} = M \smallsetminus F_i^{+N}$. Hence, for each $(c, c') \in X' \times X'$, according to Lemma~\ref{lem:global-transition-function-and-interior-closure}, we have $\Delta(c)\restriction_{M \smallsetminus F_i^{+N}} = \Delta(c')\restriction_{M \smallsetminus F_i^{+N}}$. Therefore, 
    \begin{equation*}
      \abs{\Delta(X')} =    \abs{\pi_{F_i^{+N}}(\Delta(X'))}
                       \leq \abs{\pi_{F_i^{+N}}(\Delta(Q^M))}
                       =    \abs{\pi_{F_i^{+N}}(X)}
                       <    \abs{X'}.
    \end{equation*}
    Hence, there are $c$, $c' \in X'$ such that $c \neq c'$ and $\Delta(c) = \Delta(c')$. Thus, because $\diff(c, c') \subseteq F_i$ is finite, the map $\Delta$ is not pre-injective. \qed
  \end{proof}

  %

%

  In the proof of Theorem~\ref{thm:not-pre-injective-implies-less-entropy}, the statement of Lemma~\ref{lem:exchanging-pattern-by-other-pattern-with-same-image-yields-configuration-with-same-image} is essential, which says that if two distinct patterns have the same image and we replace each occurrence of the first by the second in a configuration, we get a new configuration in which the first pattern does not occur and that has the same image as the original one.


  \begin{definition}
    Identify $M$ with $G \quotient G_0$ by $\iota \from m \mapsto G_{m_0, m}$. Let
    \begin{equation*}
      \inducedrightsemiaction \from M \times \bigcup_{A \subseteq M} Q^A \to     \bigcup_{A \subseteq M} Q^A, \quad
                                                                  (m, p) \mapsto \left[
                                                                                    \begin{aligned}
                                                                                      m \rightsemiaction \domain(p) &\to     Q,\\
                                                                                               m \rightsemiaction a &\mapsto p(a).
                                                                                    \end{aligned}
                                                                                  \right]
    \end{equation*}
  \end{definition}

  \begin{remark} 
    Let $A$ be a subset of $M$, let $p$ be map from $A$ to $Q$, and let $m$ be an element of $M$. Then, $m \inducedrightsemiaction p = g_{m_0, m} \inducedleftaction p$. 
  \end{remark}

  \begin{definition} 
    Identify $M$ with $G \quotient G_0$ by $\iota \from m \mapsto G_{m_0, m}$, let $A$ be a subset of $M$, let $p$ be map from $A$ to $Q$, let $c$ be map from $M$ to $Q$, let $m$ be an element of $M$. The pattern $p$ is said to \define{occur at $m$ in $c$}\graffito{$p$ occurs at $m$ in $c$} and we write $p \occurs_m c$\graffito{$p \occurs_m c$} if and only if
      $m \inducedrightsemiaction p = c\restriction_{m \rightsemiaction A}$.
  \end{definition}

  \begin{lemmax}{X} 
  \label{lem:m-liberation-A-is-either-outside-of-A-or-inside-of-closure-of-A}
    Let $A$ be a subset of $M$, and let $E$ and $E'$ be two subsets of $G \quotient G_0$ such that
      $\set{g^{-1} \cdot e' \suchthat e, e' \in E, g \in e} \subseteq E'$. 
    For each element $m \in M$, we have $m \rightsemiaction E \subseteq M \smallsetminus A$ or $m \rightsemiaction E \subseteq A^{+E'}$.
  \end{lemmax}

  \begin{proof}
%
    Let $m \in M$ such that $m \rightsemiaction E \nsubseteq M \smallsetminus A$. Then, $(m \rightsemiaction E) \cap A \neq \emptyset$. Hence, there is an $e' \in E$ such that $m \rightsemiaction e' \in A$. Let $e \in E$. According to Lemma~\ref{lem:rightsemiaction-can-be-undone}, there is a $g \in e$ such that $(m \rightsemiaction e) \rightsemiaction g^{-1} \cdot e' = m \rightsemiaction e'$. Because $g^{-1} \cdot e' \in E'$ and $m \rightsemiaction e' \in A$, we have $(m \rightsemiaction e) \rightsemiaction E' \cap A \neq \emptyset$. Thus, $m \rightsemiaction e \in A^{+E'}$. Therefore, $m \rightsemiaction E \subseteq A^{+E'}$. \qed 
  \end{proof}

  \begin{lemma} 
  \label{lem:exchanging-pattern-by-other-pattern-with-same-image-yields-configuration-with-same-image}
    Identify $M$ with $G \quotient G_0$ by $\iota \from m \mapsto G_{m_0, m}$, let $A$ be a subset of $M$, let $N'$ be the subset $\set{g^{-1} \cdot n' \suchthat n, n' \in N, g \in n}$ of $G \quotient G_0$, and let $p$ and $p'$ be two maps from $A^{+N'}$ to $Q$ such that $p\restriction_{A^{+N'} \smallsetminus A} = p'\restriction_{A^{+N'} \smallsetminus A}$ and $\Delta_{A^{+N'}}^-(p) = \Delta_{A^{+N'}}^-(p')$. Furthermore, let $c$ be a map from $M$ to $Q$ and let $S$ be a subset of $M$, such that the family $\family{s \rightsemiaction A^{+N'}}_{s \in S}$ is pairwise disjoint and, for each cell $s \in S$, we have $p \occurs_s c$. Put
    \begin{equation*} 
      c' = c\restriction_{M \smallsetminus (\bigcup_{s \in S} s \rightsemiaction A^{+N'})} \times \coprod_{s \in S} s \inducedrightsemiaction p'.
    \end{equation*}
    Then, for each cell $s \in S$, we have $p' \occurs_s c'$, and $\Delta(c) = \Delta(c')$. In particular, if $p \neq p'$, then, for each cell $s \in S$, we have $p \not\occurs_s c'$.
  \end{lemma} 

  \begin{proof}
    For each $s \in S$, we have $\domain(s \inducedrightsemiaction p) = \domain(s \inducedrightsemiaction p') = s \rightsemiaction A^{+N'}$. Hence, $c'$ is well-defined. Moreover, for each $s \in S$, we have $(s \inducedrightsemiaction p)\restriction_{(s \rightsemiaction A^{+N'}) \smallsetminus (s \rightsemiaction A)} = (s \inducedrightsemiaction p')\restriction_{(s \rightsemiaction A^{+N'}) \smallsetminus (s \rightsemiaction A)}$.

    Let $m \in M \smallsetminus (\bigcup_{s \in S} s \rightsemiaction A)$. If $m \in M \smallsetminus (\bigcup_{s \in S} s \rightsemiaction A^{+N'})$, then $c'(m) = c(m)$. And, if there is an $s \in S$ such that $m \in s \rightsemiaction A^{+N'}$, then, because $m \notin s \rightsemiaction A$, we have $c'(m) = (s \inducedrightsemiaction p')(m) = (s \inducedrightsemiaction p)(m) = c(m)$. Therefore,
    \begin{equation*}
      c' = c\restriction_{M \smallsetminus (\bigcup_{s \in S} s \rightsemiaction A)} \times \coprod_{s \in S} s \inducedrightsemiaction (p'\restriction_A).
    \end{equation*}

    Let $m \in M$.
    \begin{description}
      \item[Case 1:] $m \rightsemiaction N \subseteq M \smallsetminus (\bigcup_{s \in S} s \rightsemiaction A)$. Then, $c'\restriction_{m \rightsemiaction N} = c\restriction_{m \rightsemiaction N}$. Hence, $\Delta(c')(m) = \Delta(c)(m)$.
      \item[Case 2:] $m \rightsemiaction N \nsubseteq M \smallsetminus (\bigcup_{s \in S} s \rightsemiaction A)$. Then, there is an $s \in S$ such that $m \rightsemiaction N \nsubseteq M \smallsetminus (s \rightsemiaction A)$. Thus, according to Lemma~\ref{lem:m-liberation-A-is-either-outside-of-A-or-inside-of-closure-of-A}, we have $m \rightsemiaction N \subseteq (s \rightsemiaction A)^{+N'}$. Hence, because $G_0 \cdot N' \subseteq N'$, according to Item~\ref{it:properties-of-interior-closure-and-boundary:commute-with-liberation} of Lemma~\ref{lem:properties-of-interior-closure-and-boundary}, we have $m \rightsemiaction N \subseteq s \rightsemiaction A^{+N'}$ and hence $m \in (s \rightsemiaction A^{+N'})^{-N}$. Therefore, because $c\restriction_{s \rightsemiaction A^{+N'}} = s \inducedrightsemiaction p$, $\Delta_{A^{+N'}}^-(p) = \Delta_{A^{+N'}}^-(p')$, and $c'\restriction_{s \rightsemiaction A^{+N'}} = s \inducedrightsemiaction p'$, 
            \begin{align*}
              \Delta(c)(m)
              = \Delta_{A^{+N'}}^-(p)(m)
              = \Delta_{A^{+N'}}^-(p')(m)
              = \Delta(c')(m).
            \end{align*}
    \end{description}
    In either case, $\Delta(c)(m) = \Delta(c')(m)$. Therefore, $\Delta(c) = \Delta(c')$. \qed
  \end{proof}

  \begin{theorem} 
  \label{thm:not-pre-injective-implies-less-entropy}
    Let $\delta$ be $\bullet$-invariant, let $Q$ contain at least two elements, and let $\Delta$ not be pre-injective. Then, $\entropy_{\mathcal{F}}(\Delta(Q^M)) < \log\abs{Q}$.
  \end{theorem}

  \begin{proof}
    Suppose, without loss of generality, that $G_0 \in N$. Identify $M$ with $G \quotient G_0$ by $\iota \from m \mapsto G_{m_0, m}$.

    Because $\Delta$ is not pre-injective, there are $c$, $c' \in Q^M$ such that $\diff(c, c')$ is finite, $\Delta(c) = \Delta(c')$, and $c \neq c'$. Put $A = \diff(c, c')$, put $N' = \set{g^{-1} \cdot n' \suchthat n, n' \in N, g \in n}$, put $E = A^{+N'}$, and put $p = c\restriction_E$ and $p' = c'\restriction_E$. Because $\Delta(c) = \Delta(c')$, we have $\Delta_{A^{+N'}}^-(p) = \Delta_{A^{+N'}}^-(p')$. 

    Because $N$ is finite and, for each $n \in N$, we have $\abs{n} = \abs{G_0} < \infty$, the set $N'$ is finite. Moreover, $G_0 \cdot N' \subseteq N'$. According to Item~\ref{it:properties-of-interior-closure-and-boundary:neutral-element} of Lemma~\ref{lem:properties-of-interior-closure-and-boundary}, because $G_0 \in N'$ and $A \neq \emptyset$, we have $E \supseteq A$ and hence $E$ is non-empty. According to Item~\ref{it:properties-of-interior-closure-and-boundary:finite} of Lemma~\ref{lem:properties-of-interior-closure-and-boundary}, because $G_0$, $A$, and $N'$ are finite, so is $E$. Because $E$ is non-empty, according to Theorem~\ref{thm:existence-of-tiling}, there is a subset $E'$ of $G \quotient G_0$ and an $\ntuple{E, E'}$-tiling $T$ of $\mathcal{R}$. Because $G_0$ and $E$ are non-empty and finite, so is $E'$.

    Let
      $Y = \set{y \in Q^M \suchthat \ForEach t \in T \Holds p \not\occurs_t y}$. 
    For each $t \in T$, we have $t \inducedrightsemiaction p \notin \pi_{t \rightsemiaction E}(Y)$ and therefore $\pi_{t \rightsemiaction E}(Y) \subsetneqq Q^{t \rightsemiaction E}$. According to Lemma~\ref{lem:entorpy-bounded-above-if-strange-tiling-exists}, we have $\entropy_{\mathcal{F}}(Y) < \log\abs{Q}$. Hence, according to Theorem~\ref{thm:entropy-does-non-increase}, we have $\entropy_{\mathcal{F}}(\Delta(Y)) < \log\abs{Q}$.

    Let $x \in Q^M$. Put $S = \set{t \in T \suchthat p \occurs_t x}$. According to Lemma~\ref{lem:exchanging-pattern-by-other-pattern-with-same-image-yields-configuration-with-same-image}, there is an $x' \in Q^M$ such that $x' \in Y$ and $\Delta(x) = \Delta(x')$. Therefore, $\Delta(Q^M) = \Delta(Y)$. In conclusion, $\entropy_{\mathcal{F}}(Q^M) < \log\abs{Q}$. \qed
  \end{proof}

  \begin{main-theorem}[Garden of Eden theorem; Edward Forrest Moore, 1962; John R. Myhill, 1963] 
  \label{thm:garden-of-eden} 
    Let $\mathcal{M} = \ntuple{M, G, \leftaction}$ be a right amenable left homogeneous space with finite stabilisers and let $\Delta$ be the global transition function of a cellular automaton over $\mathcal{M}$ with finite set of states and finite neighbourhood. The map $\Delta$ is surjective if and only if it is pre-injective. 
  \end{main-theorem} 

  \begin{proof}
    There is a coordinate system $\mathcal{K} = \ntuple{m_0, \family{g_{m_0, m}}_{m \in M}}$ such that the cell space $\mathcal{R} = \ntuple{\mathcal{M}, \mathcal{K}}$ is right amenable. Moreover, according to \cite[Theorem~1]{wacker:automata:2016}, there is a cellular automaton $\mathcal{C} = \ntuple{\mathcal{R}, Q, N, \delta}$ such that $Q$ and $N$ are finite and $\Delta$ is its global transition function.

    \begin{description}
      \item[Case $\abs{Q} \leq 1$.] If $\abs{Q} = 0$, then, because $\abs{M} \neq 0$, we have $\abs{Q^M} = 0$. And, if $\abs{Q} = 1$, then $\abs{Q^M} = 1$. In either case, $\Delta$ is bijective, in particular, surjective and pre-injective.
      \item[Case $\abs{Q} \geq 2$.]
            According to Theorem~\ref{thm:not-surjective-implies-less-entropy} and Item~\ref{it:entropy-basic-facts:whole-space} of Lemma~\ref{lem:entropy-basic-facts}, the map $\Delta$ is not surjective if and only if $\entropy_{\mathcal{F}}(\Delta(Q^M)) < \log\abs{Q}$. And, according to Theorem~\ref{thm:less-entropy-implies-not-pre-injective} and Theorem~\ref{thm:not-pre-injective-implies-less-entropy}, we have $\entropy_{\mathcal{F}}(\Delta(Q^M)) < \log\abs{Q}$ if and only if $\Delta$ is not pre-injective. Hence, $\Delta$ is not surjective if and only if it is not pre-injective. In conclusion, $\Delta$ is surjective if and only if it is pre-injective. \qed
    \end{description}
  \end{proof}

  \begin{remark}
    In the situation of Remark~\ref{rem:group:interior-closure-boundary}, Main Theorem~\ref{thm:garden-of-eden} is \cite[Theorem~5.3.1]{ceccherini-silberstein:coornaert:2010}.
  \end{remark}

  \newpage

  \appendix

  In Appendix~\ref{apx:proof-ideas} we present some non-rigorous proof ideas of important lemmata and theorems. And in Appendix~\ref{apx:topologies-and-nets} we present the basic theory of topologies and nets.

  \section{Proof Ideas}
  \label{apx:proof-ideas}

  \begin{proof-idea}[Lemma~\ref{lem:upper-bound-of-tiling-cap-folner-net}; See Fig.~\ref{fig:upper-bound-of-tiling-cap-folner-net}]
    Let $T_i^- = T \cap F_i^{-E}$ and let $T_i^+ = T \cap F_i^{+E'}$. Then,
    \begin{equation*}
      \abs{F_i} \leq \abs{T_i^+} \cdot \abs{E'}
                \leq \parens*{\abs{T_i^-} + \abs{\boundary_{E'} F_i}} \cdot \abs{E'}
    \end{equation*}
    and hence
    \begin{equation*}
      \abs{T_i^-} \geq \abs{F_i} \cdot \parens*{\frac{1}{\abs{E'}} - \frac{\abs{\boundary_{E'} F_i}}{\abs{F_i}}}.
    \end{equation*}
    For great enough indices $i \in I$, the right side gets arbitrarily close to $\abs{F_i} / \abs{E'}$.
  \end{proof-idea}

  \begin{intuition}[Definition~\ref{def:entropy}] 
    The entropy $\entropy_{\mathcal{F}}(X)$ of $X$ with respect to the right Følner net $\mathcal{F}$ in $\mathcal{R}$ is the asymptotic growth rate of the number of finite patterns with domain $F_i$ that occur in $X$, that is, 
    \begin{equation*}
      \net{2^{\abs{F_i} \cdot \entropy_{\mathcal{F}}(X)}}_{i \in I} \sim \net{\abs{\pi_{F_i}(X)}}_{i \in I}, 
    \end{equation*}
    where $\sim$ is the binary relation, read \emph{asymptotic to}, given by
    \begin{equation*}
      \ForEach \net{r_i}_{i \in I} \ForEach \net{r_i'}_{i \in I} \Holds \net{r_i}_{i \in I} \sim \net{r_i'}_{i \in I} \iff \lim_{i \in I} \frac{r_i}{r_i'} = 1.
    \end{equation*}
  \end{intuition} 

  \begin{proof-idea}[Theorem~\ref{thm:entropy-does-non-increase}]
    Because of the locality of $\Delta$, we have $\abs{\pi_{F_i^{-N}}(\Delta(X))} \leq \abs{\pi_{F_i}(X)}$ and hence, for each index $i \in I$,
    \begin{equation*}
      \abs{\pi_{F_i}(\Delta(X))} \leq \abs{\pi_{F_i^{-N}}(\Delta(X))} \cdot \abs{Q}^{\abs{F_i \smallsetminus F_i^{-N}}}
                                 \leq \abs{\pi_{F_i}(X)} \cdot \abs{Q}^{\abs{F_i \smallsetminus F_i^{-N}}}.
    \end{equation*}
    And, because $\mathcal{F}$ is a right Følner net, $\net{\abs{F_i \smallsetminus F_i^{-N}}}_{i \in I} \sim \net{0 \cdot \abs{F_i}}_{i \in I}$ and hence 
      $\net{\abs{Q}^{\abs{F_i \smallsetminus F_i^{-N}}}}_{i \in I} \sim \net{1}_{i \in I}$.
    Thus,
    \begin{align*}
      \net{2^{\abs{F_i} \cdot \entropy_{\mathcal{F}}(\Delta(X))}}_{i \in I}
      &\sim \net{\abs{\pi_{F_i}(\Delta(X))}}_{i \in I}\\
      &\curlyeqprec \net{\abs{\pi_{F_i}(X)}}_{i \in I} \cdot \net{\abs{Q}^{\abs{F_i \smallsetminus F_i^{-N}}}}_{i \in I}\\ 
      &\sim \net{2^{\abs{F_i} \cdot \entropy_{\mathcal{F}}(X)}}_{i \in I},
    \end{align*}
    where $\curlyeqprec$ is the binary relation, read \emph{asymptotic not greater than}, given by 
    \begin{equation*}
      \ForEach \net{r_i}_{i \in I} \ForEach \net{r_i'}_{i \in I} \Holds \net{r_i}_{i \in I} \curlyeqprec \net{r_i'}_{i \in I} \iff \lim_{i \in I} \frac{r_i}{r_i'} \leq 1.
    \end{equation*}
    Therefore, $\entropy_{\mathcal{F}}(\Delta(X)) \leq \entropy_{\mathcal{F}}(X)$.
  \end{proof-idea}

  \begin{proof-idea}[Lemma~\ref{lem:entorpy-bounded-above-if-strange-tiling-exists}; See Fig.~\ref{fig:entorpy-bounded-above-if-strange-tiling-exists}]
    Let $\varepsilon \in \R_{> 0}$ and let $i_0 \in I$ be the ones from Lemma~\ref{lem:upper-bound-of-tiling-cap-folner-net}. Furthermore, let $T_i = T \cap F_i^{-E}$ and let $F_i^* = F_i \smallsetminus (\bigcup_{t \in T_i} t \rightsemiaction E)$. Moreover, let $c = - \log (1 - \frac{1}{\abs{Q^E}}) \in \R_{> 0}$. Then, for each index $i \geq i_0$,
    \begin{align*}
      \abs{\pi_{F_i}(X)} &\leq \abs{Q^{F_i^*}} \cdot \prod_{t \in T_i} \abs{\pi_{t \rightsemiaction E}(X)}\\
                         &\leq \frac{\abs{Q^{F_i}}}{\prod_{t \in T_i} \abs{Q^{t \rightsemiaction E}}} \cdot \prod_{t \in T_i} \parens*{\abs{Q^{t \rightsemiaction E}} - 1}\\
                         &=    \abs{Q^{F_i}} \cdot \prod_{t \in T_i} \parens*{1 - \frac{1}{\abs{Q^{t \rightsemiaction E}}}}\\
                         &=    \abs{Q^{F_i}} \cdot \parens*{1 - \frac{1}{\abs{Q^E}}}^{\abs{T_i}}\\
                         &\leq \abs{Q^{F_i}} \cdot \parens*{1 - \frac{1}{\abs{Q^E}}}^{\varepsilon \abs{F_i}}\\
                         &=    \abs{\pi_{F_i}(Q^M)} \cdot 2^{\abs{F_i} \cdot (- c \varepsilon)}.
    \end{align*}
    Hence, because $c \varepsilon > 0$,
    \begin{align*}
      \net{2^{\abs{F_i} \cdot \entropy_{\mathcal{F}}(X)}}_{i \in I}
      &\sim \net{\abs{\pi_{F_i}(X)}}_{i \in I}\\
      &\curlyeqprec \net{2^{\abs{F_i} \cdot \entropy_{\mathcal{F}}(\Delta(Q^M))}}_{i \in I} \cdot \net{2^{\abs{F_i} \cdot (- c \varepsilon)}}_{i \in I}\\
      &\prec \net{2^{\abs{F_i} \cdot \log\abs{Q}}}_{i \in I},
    \end{align*}
    where $\prec$ is the binary relation, read \emph{asymptotic less than}, given by
    \begin{equation*}
      \ForEach \net{r_i}_{i \in I} \ForEach \net{r_i'}_{i \in I} \Holds \net{r_i}_{i \in I} \curlyeqprec \net{r_i'}_{i \in I} \iff \lim_{i \in I} \frac{r_i}{r_i'} < 1.
    \end{equation*}
    Therefore, $\entropy_{\mathcal{F}}(X) < \log\abs{Q}$.
  \end{proof-idea}

  \begin{proof-idea}[Lemma~\ref{lem:entropy-invariant-under-closure-net-change}] 
    We have $\abs{\pi_{F_i^{+E}}(X)} \leq \abs{\pi_{F_i}(X)} \cdot \abs{Q}^{\abs{F_i^{+E} \smallsetminus F_i}}$. And, because $\mathcal{F}$ is a right Følner net, we have $\net{\abs{Q^{F_i^{+E} \smallsetminus F_i}}}_{i \in I} \sim \net{1}_{i \in I}$. Thus, 
    \begin{align*}
      \net{2^{\abs{F_i^{+E}} \cdot \entropy_{\net{F_i^{+E}}_{i \in I}}(X)}}_{i \in I}
      &\sim \net{\abs{\pi_{F_i^{+E}}(X)}}_{i \in I}\\
      &\curlyeqprec \net{\abs{\pi_{F_i}(X)}}_{i \in I} \cdot \net{\abs{Q}^{\abs{F_i^{+E} \smallsetminus F_i}}}_{i \in I}\\ 
      &\sim \net{2^{\abs{F_i} \cdot \entropy_{\mathcal{F}}(X)}}_{i \in I}.
    \end{align*}
    Therefore, $\entropy_{\net{F_i^{+E}}_{i \in I}}(X) \leq \entropy_{\mathcal{F}}(X)$.
  \end{proof-idea}

  \begin{proof-idea}[Theorem~\ref{thm:less-entropy-implies-not-pre-injective}] 
    The asymptotic growth rate of finite patterns in $\Delta(Q^M)$ is less than the one of $Q^M$. Hence, there are at least two finite patterns in $Q^M$ with a domain $A$ that have the same image under $\Delta_A^{-}$.
  \end{proof-idea}

  \begin{proof-idea}[Theorem~\ref{thm:not-pre-injective-implies-less-entropy}] 
    For $N' = \bigcup_{n \in N} n^{-1} N$, there is a subset $A$ of $M$ and there are two distinct finite patterns $p$ and $p'$ with domain $A^{+N'}$ that have the same image under $\Delta_{A^{+N'}}^-$. The set $Y$ of all configurations in which $p$ does not occur at the cells of a tiling has the same image under $\Delta$ as $Q^M$, because in a configuration we may replace occurrences of $p$ by $p'$ without changing the image. Thus, $\entropy(\Delta(Q^M)) = \entropy(\Delta(Y)) \leq \entropy(Y)$. Moreover, because $Y$ is missing the pattern $p$ at each cell of a tiling, we have $\entropy(Y) < \entropy(Q^M) = \log\abs{Q}$.
  \end{proof-idea}

  \section{Topologies and Nets}
  \label{apx:topologies-and-nets}

  The theory of topologies and nets as presented here may be found in more detail in Appendix~A in the monograph \enquote{Cellular Automata and Groups}\cite{ceccherini-silberstein:coornaert:2010}.

  \begin{definition}
    Let $X$ be a set and let $\mathcal{T}$ be a set of subsets of $X$. The set $\mathcal{T}$ is called \define{topology on $X$}\graffito{topology $\mathcal{T}$ on $X$} if, and only if
    \begin{enumerate}
      \item $\set{\emptyset, X}$ is a subset of $\mathcal{T}$,
      \item for each family $\family{O_i}_{i \in I}$ of elements in $\mathcal{T}$, the union $\bigcup_{i \in I} O_i$ is an element of $\mathcal{T}$,
      \item for each finite family $\family{O_i}_{i \in I}$ of elements in $\mathcal{T}$, the intersection $\bigcap_{i \in I} O_i$ is an element of $\mathcal{T}$. 
    \end{enumerate}
  \end{definition}

  \begin{definition}
    Let $X$ be a set, and let $\mathcal{T}$ and $\mathcal{T}'$ be two topologies on $X$. The topology $\mathcal{T}$ is called
    \begin{enumerate}
      \item \define{coarser than $\mathcal{T}'$}\graffito{coarser than $\mathcal{T}'$} if, and only if $\mathcal{T} \subseteq \mathcal{T}'$;
      \item \define{finer than $\mathcal{T}'$}\graffito{finer than $\mathcal{T}'$} if, and only if $\mathcal{T} \supseteq \mathcal{T}'$.
    \end{enumerate}
  \end{definition}

  \begin{definition} 
    Let $X$ be a set and let $\mathcal{T}$ be a topology on $X$. The tuple $(X, \mathcal{T})$ is called \define{topological space}\graffito{topological space $(X, \mathcal{T})$}, each subset $O$ of $X$ with $O \in \mathcal{T}$ is called \define{open in $X$}\graffito{open set $O$ in $X$}, each subset $A$ of $X$ with $X \smallsetminus A \in \mathcal{T}$ is called \define{closed in $X$}\graffito{closed set $A$ in $X$}, and each subset $U$ of $X$ that is both open and closed is called \define{clopen in $X$}\graffito{clopen set $U$ in $X$}. 

    The set $X$ is said to be \define{equipped with $\mathcal{T}$}\graffito{equipped with $\mathcal{T}$} if, and only if it shall be implicitly clear that $\mathcal{T}$ is the topology on $X$ being considered. The set $X$ is called \define{topological space}\graffito{topological space $X$} if, and only if it is implicitly clear what topology on $X$ is being considered. 
  \end{definition}

  \begin{example}
    Let $X$ be a set. The set $\powerset(X)$ is the finest topology on $X$. Itself as well as the topological space $(X, \powerset(X))$ are called \define{discrete}\index{discrete!topological space}\index{discrete!topology}\graffito{discrete}.
  \end{example}

  \begin{definition} 
    Let $(X, \mathcal{T})$ be a topological space, let $x$ be a point of $X$, and let $N$ be a subset of $X$. The set $N$ is called \define{neighbourhood of $x$}\graffito{neighbourhood of $x$} if, and only if there is an open subset $O$ of $X$ such that $x \in O$ and $O \subseteq N$.
  \end{definition}

  \begin{definition} 
    Let $(X, \mathcal{T})$ be a topological space and let $x$ be a point of $X$. The set of all open neighbourhoods of $x$ is denoted by $\mathcal{T}_x$. 
  \end{definition}

  \begin{definition} 
    Let $I$ be a set and let $\leq$ be a binary relation on $I$.
    The relation $\leq$ is called \define{preorder on $I$}\graffito{preorder $\leq$ on $I$} and the tuple $(I, \leq)$ is called \define{preordered set}\graffito{preordered set $(I, \leq)$} if, and only if the relation $\leq$ is reflexive and transitive.
  \end{definition}

  \begin{definition}
    Let $\leq$ be a preorder on $I$. It is called \define{directed}\graffito{directed preorder $\leq$ on $I$}\index{preorder on $I$!directed} and the preordered set $(I, \leq)$ is called \define{directed set}\graffito{directed set $(I, \leq)$} if, and only if
    \begin{equation*}
      \ForEach i \in I \ForEach i' \in I \Exists i'' \in I \SuchThat i \leq i'' \land i' \leq i''.
    \end{equation*}
  \end{definition}

  \begin{definition} 
    Let $\leq$ be a preorder on $I$, let $J$ be a subset of $I$, and let $i$ be an element of $I$. The element $i$ is called \define{upper bound of $J$ in $(I, \leq)$}\graffito{upper bound of $J$ in $(I, \leq)$} if, and only if
    \begin{equation*}
      \ForEach i' \in J \Holds i' \leq i.
    \end{equation*}
  \end{definition}

  \begin{definition} 
    Let $M$ be a set, let $I$ be a set, and let $f \from I \to M$ be a map. The map $f$ is called \define{family of elements in $M$ indexed by $I$}\graffito{family $\family{m_i}_{i \in I}$ of elements in $M$ indexed by $I$} and denoted by $\family{m_i}_{i \in I}$, where, for each index $i \in I$, $m_i = f(i)$.
  \end{definition}

  \begin{definition}
    Let $I$ be a set, let $\leq$ be a binary relation on $I$, and let $\family{m_i}_{i \in I}$ be a family of elements in $M$ indexed by $I$. The family $\family{m_i}_{i \in I}$ is called \define{net in $M$ indexed by $(I, \leq)$}\graffito{net $\family{m_i}_{i \in I}$ in $M$ indexed by $(I, \leq)$} if, and only if the tuple $(I, \leq)$ is a directed set.
  \end{definition}

  \begin{definition} 
    Let $\net{m_i}_{i \in I}$ and $\net{m_j'}_{j \in J}$ be two nets in $M$. The net $\net{m_j'}_{j \in J}$ is called \define{subnet of $\net{m_i}_{i \in I}$}\graffito{subnet $\net{m_j'}_{j \in J}$ of $\net{m_i}_{i \in I}$} if, and only if there is a map $f \from J \to I$ such that $\net{m_j'}_{j \in J} = \net{m_{f(j)}}_{j \in J}$ and
    \begin{equation*} 
      \ForEach i \in I \Exists j \in J \SuchThat \ForEach j' \in J \Holds (j' \geq j \implies f(j') \geq i).
    \end{equation*}
  \end{definition}

  \begin{definition}
    Let $(X, \mathcal{T})$ be a topological space, let $\net{x_i}_{i \in I}$ be a net in $X$ indexed by $(I, \leq)$, and let $x$ be a point of $X$. The net $\net{x_i}_{i \in I}$ is said to \define{converge to $x$}\graffito{converge to $x$} and $x$ is called \define{limit point of $\net{x_i}_{i \in I}$}\graffito{limit point of $\net{x_i}_{i \in I}$} if, and only if
    \begin{equation*}
      \ForEach O \in \mathcal{T}_x \Exists i_0 \in I \SuchThat \ForEach i \in I \Holds (i \geq i_0 \implies x_i \in O).
    \end{equation*}
  \end{definition} 

  \begin{definition}
    Let $(X, \mathcal{T})$ be a topological space and let $\net{x_i}_{i \in I}$ be a net in $X$ indexed by $(I, \leq)$. The net $\net{x_i}_{i \in I}$ is called \define{convergent}\graffito{convergent} if, and only if there is a point $x \in X$ such that it converges to $x$.
  \end{definition}

  \begin{remark} 
    Let $\net{m_i}_{i \in I}$ be a net that converges to $x$. Each subnet $\net{m_j'}_{j \in J}$ of $\net{m_i}_{i \in I}$ converges to $x$.
  \end{remark}

  \begin{lemma} 
    Let $(X, \mathcal{T})$ be a topological space, let $Y$ be a subset of $X$, and let $x$ be an element of $X$. Then, $x \in \closure{Y}$ if, and only if there is a net $\net{y_i}_{i \in I}$ in $Y$ that converges to $x$. 
  \end{lemma}

  \begin{proof}
    See Proposition~A.2.1 in \enquote{Cellular Automata and Groups}\cite{ceccherini-silberstein:coornaert:2010}. \qed
  \end{proof}

  \begin{lemma} 
    Let $(X, \mathcal{T})$ be a topological space. It is Hausdorff if, and only if each convergent net in $X$ has exactly one limit point. 
  \end{lemma}

  \begin{proof}
    See Proposition~A.2.2 in \enquote{Cellular Automata and Groups}\cite{ceccherini-silberstein:coornaert:2010}. \qed
  \end{proof}

  \begin{definition}
    Let $(X, \mathcal{T})$ be a Hausdorff topological space, let $\net{x_i}_{i \in I}$ be a convergent net in $X$ indexed by $(I, \leq)$, and let $x$ be the limit point of $\net{x_i}_{i \in I}$. The point $x$ is denoted by $\lim_{i \in I} x_i$\graffito{$\lim_{i \in I} x_i$} and we write $x_i \underset{i \in I}{\to} x$\graffito{$x_i \underset{i \in I}{\to} x$}.
  \end{definition}

  \begin{definition}
    Let $(X, \mathcal{T})$ be a topological space, let $\net{x_i}_{i \in I}$ be a net in $X$ indexed by $(I, \leq)$, and let $x$ be an element of $X$. The point $x$ is called \define{cluster point of $\net{x_i}_{i \in I}$}\graffito{cluster point $x$ of $\net{x_i}_{i \in I}$} if, and only if
    \begin{equation*}
      \ForEach O \in \mathcal{T}_x \ForEach i \in I \Exists i' \in I \SuchThat (i' \geq i \land x_{i'} \in O).
    \end{equation*}
  \end{definition}

  \begin{lemma} 
    Let $(X, \mathcal{T})$ be a topological space, let $\net{x_i}_{i \in I}$ be a net in $X$ indexed by $(I, \leq)$, and let $x$ be an element of $X$. The point $x$ is a cluster point of $\net{x_i}_{i \in I}$ if, and only if there is a subnet of $\net{x_i}_{i \in I}$ that converges to $x$.
  \end{lemma}

  \begin{proof}
    See Proposition~A.2.3 in \enquote{Cellular Automata and Groups}\cite{ceccherini-silberstein:coornaert:2010}. \qed
  \end{proof}

  \begin{lemma}
    Let $(X, \mathcal{T})$ and $(X', \mathcal{T}')$ be two topological spaces, let $f$ be a continuous map from $X$ to $X'$, let $\net{x_i}_{i \in I}$ be a net in $X$, and let $x$ be an element of $X$.
    \begin{enumerate}
      \item If $x$ is a limit point of $\net{x_i}_{i \in I}$, then $f(x)$ is a limit point of $\net{f(x_i)}_{i \in I}$. 
      \item If $x$ is a cluster point of $\net{x_i}_{i \in I}$, then $f(x)$ is a cluster point of $\net{f(x_i)}_{i \in I}$.
    \end{enumerate}
  \end{lemma}

  \begin{proof} 
    Confer the last paragraph of Sect.~A.2 in \enquote{Cellular Automata and Groups}\cite{ceccherini-silberstein:coornaert:2010}. \qed
  \end{proof}

  \begin{definition} 
    Let $\overline{\R} = \R \cup \set{-\infty,+\infty}$ be the affinely extended real numbers and let $\net{r_i}_{i \in I}$ be a net in $\overline{\R}$ indexed by $(I, \leq)$.
    \begin{enumerate} 
      \item The limit of the net $\net{\inf_{i' \geq i} r_{i'}}_{i \in I}$ is called \define{limit inferior of $\net{r_i}_{i \in I}$}\graffito{limit inferior $\liminf_{i \in I} r_i$ of $\net{r_i}_{i \in I}$} and denoted by $\liminf_{i \in I} r_i$.
      \item The limit of the net $\net{\sup_{i' \geq i} r_{i'}}_{i \in I}$ is called \define{limit superior of $\net{r_i}_{i \in I}$}\graffito{limit superior $\limsup_{i \in I} r_i$ of $\net{r_i}_{i \in I}$} and denoted by $\limsup_{i \in I} r_i$.
    \end{enumerate}
  \end{definition}


\begin{thebibliography}{9}
    \bibitem{ceccherini-silberstein:coornaert:2010} Ceccherini-Silberstein, Tullio and Coornaert, Michel. Cellular Automata and Groups. In: Springer Monographs in Mathematics. Springer-Verlag, 2010.
    \bibitem{moore:1962} Moore, Edward Forrest. Machine models of self-reproduction. In: Proceedings of Symposia in Applied Mathematics 14 (1962), pages 17–33.
    \bibitem{myhill:1963} Myhill, John R. The converse of Moore's Garden-of-Eden theorem. In: Proceedings of the American Mathematical Society 14 (1963), pages 685–686.
    \bibitem{moriceau:2011} Moriceau, S{\'{e}}bastien. Cellular Automata on a $G$-Set. In: Journal of Cellular Automata 6.6 (2011), pages 461-486.
    \bibitem{wacker:automata:2016} Wacker, Simon. Cellular Automata on Group Sets and the Uniform Curtis-Hedlund-Lyndon Theorem. In: Cellular Automata and Discrete Complex Systems (2016), pages 185-198. arXiv:1603.07271 [math.GR].
    \bibitem{wacker:amenable:2016} Wacker, Simon. Right Amenable Left Group Sets and the Tarski-Følner Theorem. Preprint (2016). arXiv:1603.06460 [math.GR].
  \end{thebibliography}
\end{document}